\documentclass[a4paper,reqno,10pt]{amsart}

\usepackage{epsf,graphicx,epsfig}
\usepackage{amscd}
\usepackage[all,knot]{xy}
\usepackage{amsmath,latexsym,amssymb,amsthm}
\usepackage[nospace,noadjust]{cite}
\usepackage{textcomp}
\usepackage{dsfont}
\usepackage{setspace,cite}
\usepackage{lscape,fancyhdr,fancybox}
\usepackage{stmaryrd}
\usepackage{tikz}
\usepackage{cancel}
\usetikzlibrary{shapes,arrows,decorations.markings}
\usepackage{graphicx}

\usepackage{color}
\usepackage{url}
\usepackage{enumerate}
\usepackage[mathscr]{euscript}
  \theoremstyle{plain}

\swapnumbers
    \newtheorem{thm}{Theorem}[section]
    \newtheorem{prop}[thm]{Proposition}

    \newtheorem{subsec}[thm]{}
\theoremstyle{definition}
    \newtheorem{defn}[thm]{Definition}
        \newtheorem{remark}[thm]{Remark}
    \newtheorem{exam}[thm]{Example}

\theoremstyle{remark}

\title{}
\author{}
\date{}
\usepackage{amssymb}

\usepackage{hyperref}

\hypersetup{
	colorlinks,
	citecolor=blue,
	filecolor=black,
	linkcolor=blue,
	urlcolor=black
}

\begin{document}
 
\title[]{Associative-Yamaguti algebras}

\author{Apurba Das}
\address{Department of Mathematics,
Indian Institute of Technology, Kharagpur 721302, West Bengal, India.}
\email{apurbadas348@gmail.com, apurbadas348@maths.iitkgp.ac.in}

\begin{abstract}
In this paper, we first introduce associative-Yamaguti algebras as the associative analogue of Lie-Yamaguti algebras. Associative algebras, reductive associative algebras and associative triple systems of the first kind form subclasses of associative-Yamaguti algebras. Any diassociative algebra canonically provides an associative-Yamaguti algebra structure. We confirm that any associative-Yamaguti algebra admits an enveloping associative algebra (i.e., it can be obtained from a reductive associative algebra). We show that a suitable skew-symmetrization of an associative-Yamaguti algebra gives rise to a Lie-Yamaguti algebra structure. Next, we define the $(2,3)$-cohomology group of an associative-Yamaguti algebra to study formal one-parameter deformations and abelian extensions. Later, we consider Yamaguti multiplications on a nonsymmetric operad as a generalization of associative-Yamaguti algebras. This notion further leads us to introduce dendriform-Yamaguti algebras, which are splitting objects for associative-Yamaguti algebras. Finally, we consider relative Rota-Baxter operators on associative-Yamaguti algebras to establish close relationships with dendriform-Yamaguti algebras.
\end{abstract}

\maketitle

\medskip

\begin{center}

\end{center}

\thispagestyle{empty}

\tableofcontents

\section{Introduction}
The concept of Lie-Yamaguti algebras can be implicitly traced back to the work of Nomizu \cite{nomizu} in the invariant affine connections on homogeneous spaces. Inspired by his work, K. Yamaguti introduced general Lie triple systems together with their representations and cohomology theory \cite{yamaguti0,yamaguti}. Later, Kinyon and Weinstein \cite{kinyon} observed that general Lie triple systems, which they called {\em Lie-Yamaguti algebras} in their paper, can be constructed from Leibniz algebras. Additionally, they examine the {\em enveloping Lie algebras} of Lie-Yamaguti algebras. Although Lie algebras and Lie triple systems are the first examples of Lie-Yamaguti algebras, this new structure exhibits greater complexity than either Lie algebras or Lie triple systems. This fact was observed in the works \cite{benito0,benito1,takahashi}, where the authors explained irreducible Lie-Yamaguti algebras and some other Lie-Yamaguti algebras with specific enveloping Lie algebras. Since then, Lie-Yamaguti algebras have attracted much attention and have been extensively studied. In \cite{zhang-li}, the authors considered deformations and extensions of Lie-Yamaguti algebras in terms of the cohomology introduced in \cite{yamaguti}. In \cite{sheng-complex}, the authors analyzed product structures and complex structures on Lie-Yamaguti algebras using Nijenhuis operators. On the other hand, Rota-Baxter operators and symplectic structures on Lie-Yamaguti algebras have been studied in \cite{sheng-zhao}. Recently, homomorphisms \cite{mondal-saha}, crossed homomorphisms \cite{sun-chen}, automorphisms \cite{goswami} and nilpotency \cite{hani} of Lie-Yamaguti algebras have also been investigated.

\medskip

Besides Lie triple systems and Lie-Yamaguti algebras, there exist well-known triple systems in the world of associative algebras, Jordan algebras and Leibniz algebras. See \cite{bremner, carlsson, jacobson, lister, meyberg} for more details about associative triple systems and their close connections with Lie, Jordan and Leibniz triple systems. Although {\em associative triple systems} are extensively studied in the literature from various points of view, the notion of an associative-Yamaguti algebra is not yet considered. In this paper, we first propose the definition of an {\em associative-Yamaguti algebra} and provide some interesting examples (cf. Definition \ref{defn-assy} and subsequent examples). Of course, associative algebras, reductive associative algebras and associative triple systems of the first kind are the first examples of associative-Yamaguti algebras. On the other side, diassociative algebra has its origin in the study of Leibniz algebras \cite{loday}. We show that any diassociative algebra yields an associative-Yamaguti algebra structure (cf. Theorem \ref{diass-assy}). We also confirm that any associative-Yamaguti algebra admits an enveloping associative algebra (cf. Theorem \ref{thm-env-ass}). That is, it can be embedded into a bigger reductive associative algebra so that the induced associative-Yamaguti algebra structure coincides with the given one. This concept generalizes the enveloping Lie algebra of a given Lie-Yamaguti algebra \cite{kinyon}. Next, we show that a suitable skew-symmetrization of an associative-Yamaguti algebra gives rise to a Lie-Yamaguti algebra (cf. Theorem \ref{assy-liey}) and this construction is compatible with the well-known construction of a Lie algebra from an associative algebra. In turn, we obtain the commutative diagram given in (\ref{ass-diagram}). This is also compatible with the constructions of an associative-Yamaguti algebra from a diassociative algebra and of a Lie-Yamaguti algebra from a Leibniz algebra. More precisely, the diagram (\ref{diass-diagram}) commutes.

\medskip

Next, we focus on the cohomology of an associative-Yamaguti algebra and its potential applications. We begin with {\em representations} of an associative-Yamaguti algebra and construct the corresponding semidirect product. Subsequently, we define the {\em $(2,3)$-th cohomology group} $\mathcal{H}^{(2,3)} (A, M)$ of an associative-Yamaguti algebra $A$ with coefficients in a given representation $M$. Currently, we don't have the full cochain complex of an associative-Yamaguti algebra that generalizes the $(2,3)$-th cohomology group. Then we consider {\em formal deformations} of an associative-Yamaguti algebra $A$ and show that the corresponding infinitesimals are precisely $(2,3)$-cocycles of $A$ with coefficients in the adjoint representation (cf. Theorem \ref{thm-inf-co}). Moreover, the infinitesimals corresponding to equivalent formal deformations give rise to the same element in the $(2,3)$-th cohomology group $\mathcal{H}^{(2,3)} (A, A)$. Finally, we also consider {\em abelian extensions} of an associative-Yamaguti algebra $A$ by a given representation $M$ and show that the set of all isomorphism classes of such abelian extensions can be classified by the $(2,3)$-th cohomology group $\mathcal{H}^{(2,3)} (A, M)$ (cf. Theorem \ref{thm-abelian}).

\medskip

In \cite{gers-voro}, Gerstenhaber and Voronov considered {\em multiplications} on nonsymmetric operads as a generalization of associative algebras. Such multiplications are found to be very useful in the study of associative algebras and other Loday-type algebras (i.e., those algebras in which the variables follow the same order in all the defining identities). It can be observed that an associative-Yamaguti algebra can be regarded as a Loday-type algebra. Inspired by this, we consider {\em Yamaguti multiplications} on a nonsymmetric operad as a generalization of associative-Yamaguti algebras. More precisely, there is a one-one correspondence between Yamaguti multiplications on the endomorphism operad $\mathrm{End}_A$ and associative-Yamaguti algebra structures on the vector space $A$ (cf. Theorem \ref{thm-ym-assy}). As an interesting Loday-type algebra, the notion of {\em dendriform algebra} \cite{loday} got much attention due to its connections with Rota-Baxter algebras \cite{guo}. In \cite{das-dend}, the present author explicitly wrote the nonsymmetric operad $\mathrm{Dend}_A$ whose multiplications correspond to dendriform algebra structures on the vector space $A$. In the present paper, we define the notion of a \textit{dendriform-Yamaguti algebra} as a Yamaguti multiplication on the nonsymmetric operad $\mathrm{Dend}_A$. We explicitly write the identities satisfied in a given dendriform-Yamaguti algebra (cf. Definition \ref{defn-dendy}). Any dendriform algebras and {\em dendriform triple systems} are examples of dendriform-Yamaguti algebras. We observe that a dendriform-Yamaguti algebra splits an associative-Yamaguti algebra (cf. Theorem \ref{total-assy}). To study dendriform-Yamaguti algebras further, we introduce {\em relative Rota-Baxter operators} on associative-Yamaguti algebras. We show that such an operator induces a dendriform-Yamaguti algebra structure and, conversely, any dendriform-Yamaguti algebra arises in this way (cf. Theorems \ref{thm-relrba-dendy} and \ref{thm-last}).

\medskip

The paper is organized as follows. In Section \ref{sec2}, we recall some preliminaries on Lie-Yamaguti algebras. Then we introduce and study associative-Yamaguti algebras in Section \ref{sec3}. Their relations with diassociative algebras and Lie-Yamaguti algebras are also investigated. In Section \ref{sec4}, we introduce representations and the $(2,3)$-th cohomology group of an associative-Yamaguti algebra. We use this cohomology group in Section \ref{sec5} to study formal deformations and abelian extensions of an associative-Yamaguti algebra. Finally, we discuss Yamaguti multiplications, dendriform-Yamaguti algebras and relative Rota-Baxter operators in Section \ref{sec6}.

\medskip

All vector spaces, linear maps, unadorned tensor products and wedge products are over a field ${\bf k}$ of characteristic $0$.

\section{Some background on Lie-Yamaguti algebras}\label{sec2}

In this section, we recall Lie-Yamaguti algebras, some basic examples and representations of Lie-Yamaguti algebras \cite{jacobson,kinyon,meyberg,yamaguti0,yamaguti}. 

\begin{defn}
    A {\bf Lie-Yamaguti algebra} is a triple $(\mathfrak{g}, [~, ~], \llbracket ~, ~, ~\rrbracket)$ consisting of a vector space $\mathfrak{g}$ equipped with linear maps $[~, ~] : \wedge^2 \mathfrak{g} \rightarrow \mathfrak{g}$ and $\llbracket ~, ~, ~\rrbracket : \wedge^2 \mathfrak{g} \otimes \mathfrak{g} \rightarrow \mathfrak{g}$ that satisfy the following set of identities:
\begin{equation}\label{ly1}\tag{LY1}
    [ [x, y], z] + [[y, z], x] + [[ z, x], y] + \llbracket x, y, z \rrbracket + \llbracket y, z, x \rrbracket + \llbracket z, x, y \rrbracket = 0,
\end{equation}
\begin{equation}\label{ly2}\tag{LY2}
    \llbracket [x, y], z, w \rrbracket + \llbracket [y, z], x , w \rrbracket + \llbracket [z, x], y, w \rrbracket = 0,
\end{equation}
\begin{equation}\label{ly3}\tag{LY3}
   \llbracket x, y , [z, w ] \rrbracket = [ \llbracket x, y, z \rrbracket,  w] + [z, \llbracket x, y, w \rrbracket ],  
\end{equation}
\begin{equation}\label{ly4}\tag{LY4}
   \llbracket x, y, \llbracket z, w, t \rrbracket \rrbracket = \llbracket \llbracket x, y, z \rrbracket, w, t \rrbracket + \llbracket z, \llbracket x, y, w \rrbracket, t \rrbracket + \llbracket z, w, \llbracket x, y, t \rrbracket \rrbracket,  
\end{equation}
for all $x, y, z, w, t \in \mathfrak{g}$.
\end{defn}

Let $(\mathfrak{g}, [~,~], \llbracket ~, ~, ~ \rrbracket)$ and $(\mathfrak{g}', [~,~]', \llbracket ~, ~, ~ \rrbracket')$ be two Lie-Yamaguti algebras. A {\bf homomorphism} of Lie-Yamaguti algebras from $\mathfrak{g}$ to $\mathfrak{g}'$ is a linear map $\varphi : \mathfrak{g} \rightarrow \mathfrak{g}'$ satisfying
\begin{align*}
    \varphi ([x, y]) = [ \varphi (x), \varphi (y)]'  \quad \text{and} \quad \varphi ( \llbracket x, y, z \rrbracket) = \llbracket \varphi (x), \varphi (y), \varphi (z) \rrbracket', \text{ for all } x, y, z \in \mathfrak{g}.
\end{align*}
The collection of all Lie-Yamaguti algebras and homomorphisms between them is a category, denoted by {\bf LieY}.

\begin{exam}\label{exam-lie-liey}
    Let $(\mathfrak{g}, [~,~])$ be a Lie algebra. Then $(\mathfrak{g}, [~,~], \llbracket ~, ~, ~ \rrbracket)$ is a Lie-Yamaguti algebra, where
    \begin{align*}
        \llbracket x, y, z \rrbracket := [ [x, y], z], \text{ for } x, y, z \in \mathfrak{g}.
    \end{align*}
   %% This construction yields a functor $\mathcal{F} : {\bf Lie} \rightarrow {\bf LieY}$ from the category of Lie algebras to the category of Lie-Yamaguti algebras. 
\end{exam}

\begin{exam}
    A {\bf Lie triple system} is a pair $(\mathfrak{g}, \llbracket ~, ~, ~ \rrbracket)$ of a vector space $\mathfrak{g}$ with a linear map  $\llbracket ~, ~, ~\rrbracket : \wedge^2 \mathfrak{g} \otimes \mathfrak{g} \rightarrow \mathfrak{g}$ satisfying
    \begin{align*}
        & \qquad \qquad \llbracket x, y, z \rrbracket + \llbracket y, z, x \rrbracket + \llbracket z, x, y \rrbracket = 0, \\
        & \llbracket x, y, \llbracket z, w, t \rrbracket \rrbracket = \llbracket \llbracket x, y, z \rrbracket, w, t \rrbracket + \llbracket z, \llbracket x, y, w \rrbracket, t \rrbracket + \llbracket z, w, \llbracket x, y, t \rrbracket \rrbracket, 
    \end{align*}
    for all $x, y, z , w, t \in \mathfrak{g}$. It follows that a Lie triple system $(\mathfrak{g}, \llbracket ~, ~, ~ \rrbracket)$ can be regarded as a Lie-Yamaguti algebra with the trivial bracket operation $[~,~]$.
\end{exam}

\begin{exam}\label{exam-leib-liey}
    Let $(\mathfrak{l}, \lfloor ~, ~ \rfloor)$ be a left Leibniz algebra. That is, $\mathfrak{l}$ is a vector space equipped with a linear operation $\lfloor ~, ~ \rfloor : \mathfrak{l} \otimes \mathfrak{l} \rightarrow \mathfrak{l}$ satisfying 
    \begin{align*}
        \lfloor x, \lfloor y, z \rfloor \rfloor = \lfloor \lfloor x, y \rfloor, z \rfloor + \lfloor y, \lfloor x, z \rfloor \rfloor, \text{ for } x, y, z \in \mathfrak{l}.
    \end{align*}
    For a left Leibniz algebra $(\mathfrak{l}, \lfloor ~, ~ \rfloor)$, we define linear operations $[~,~] : \wedge^2 \mathfrak{l} \rightarrow \mathfrak{l}$ and $\llbracket ~, ~, ~ \rrbracket : \wedge^2 \mathfrak{l} \otimes \mathfrak{l} \rightarrow \mathfrak{l}$ by
    \begin{align*}
        [x, y] := \lfloor x, y \rfloor - \lfloor y, x \rfloor ~~~ \text{ and } ~~~ \llbracket x, y, z \rrbracket := - \lfloor \lfloor x, y \rfloor, z \rfloor,
    \end{align*}
    for $x, y, z \in \mathfrak{l}$. Then $(\mathfrak{l}, [~,~], \llbracket ~, ~, ~ \rrbracket)$ is a Lie-Yamaguti algebra. 
    % %This construction is functorial, giving rise to a functor $\mathcal{G} : {\bf Leib} \rightarrow {\bf LieY}$ from the category {\bf Leib} of Leibniz algebras to the category of Lie-Yamaguti algebras.
\end{exam}

Let $(\mathfrak{g}, [~,~], \llbracket ~, ~, ~ \rrbracket)$ be a Lie-Yamaguti algebra. Recall that \cite{yamaguti} a {\bf representation} of this Lie-Yamaguti algebra is given by a triple $(V, \rho, \nu)$ of a vector space $V$ endowed with linear maps $\rho : \mathfrak{g} \rightarrow \mathrm{End}(V)$ and $\nu : \mathfrak{g} \otimes \mathfrak{g} \rightarrow \mathrm{End}(V)$ that satisfy the following conditions: 
\begin{equation*}
    \nu ([x, y], z) - \nu (x, z) \rho(y) + \nu (y, z) \rho(x) = 0,
\end{equation*}
\begin{equation*}
     \nu (x, [y, z]) - \rho (y) \nu (x, z) + \rho (z) \nu (x, y) = 0,
\end{equation*}
\begin{equation*}
    \rho ( \llbracket  x, y, z \rrbracket) = [ D_{\rho, \mu} (x, y), \rho (z)],
\end{equation*}
\begin{equation*}
    \nu(z, w) \nu (x, y) - \nu (y, w) \nu(x, z) - \nu (x, \llbracket  y, z, w \rrbracket) + D_{\rho, \nu} (y, z) \nu (x, w) = 0,
\end{equation*}
\begin{equation*}
     \nu ( \llbracket x, y, z \rrbracket, w) + \nu (z, \llbracket x, y, w \rrbracket) = [D_{\rho, \nu} (x, y),  \nu(z, w)],
\end{equation*}
for all $x, y, z, w \in \mathfrak{g}$. Here the linear map $D_{\rho, \nu} : \mathfrak{g} \otimes \mathfrak{g} \rightarrow \mathrm{End} (V)$ is given by
\begin{align*}
    D_{\rho, \nu} (x, y):= [\rho(x), \rho (y)] - \rho ([x, y]) - \nu({x, y}) + \nu (y, x), \text{ for } x, y \in \mathfrak{g}.
\end{align*}

The following result is a characterization of representations of a Lie-Yamaguti algebra \cite{zhang-li}.
\begin{prop}\label{prop-zhang}
    Let $(\mathfrak{g}, [~,~], \llbracket ~, ~, ~ \rrbracket)$ be a Lie-Yamaguti algebra. Suppose $V$ is a vector space endowed with linear maps $\rho : \mathfrak{g} \rightarrow \mathrm{End}(V)$ and $\nu : \mathfrak{g} \otimes \mathfrak{g} \rightarrow \mathrm{End}(V)$. Then $(V, \rho, \nu)$ is a representation of the given Lie-Yamaguti algebra if and only if the following operations
    \begin{align*}
        [(x, u), (y, v)]_\ltimes :=~& ([x, y], \rho(x) v - \rho(y) u), \\
        \llbracket (x, u), (y, v), (z, w) \rrbracket_\ltimes :=~& \big( \llbracket x, y, z \rrbracket , D_{\rho, \nu} (x, y) w + \nu (y, z) u - \nu (x, z) v \big),
    \end{align*}
    for $(x, u), (y, v), (z, w) \in \mathfrak{g} \oplus V$, makes the triple $(\mathfrak{g} \oplus V, [~,~]_\ltimes, \llbracket ~, ~, ~ \rrbracket_\ltimes)$ into a Lie-Yamaguti algebra.
\end{prop}

%In \cite{sheng-zhao} the authors introduced preLie-Yamaguti algebras while studying relative Rota-Baxter operators on Lie-Yamaguti algebras. Here we also recall this new structure.

%\begin{defn} A {\bf preLie-Yamaguti algebra} is a triple $(\mathfrak{p}, \diamond, \langle ~, ~, ~ \rangle)$ of a vector space $\mathfrak{p}$ with linear operations $\diamond : \mathfrak{p} \otimes\mathfrak{p} \rightarrow \mathfrak{p}$ and $\langle ~, ~, ~ \rangle : \mathfrak{p} \otimes \mathfrak{p} \otimes \mathfrak{p} \rightarrow \mathfrak{p}$ such that for all $x, y, z, w, t \in \mathfrak{p}$, 
%\end{defn}

%\begin{prop}
%   Let $(\mathfrak{p}, \diamond, \langle ~, ~, ~ \rangle)$ be a preLie-Yamaguti algebra. Then $(\mathfrak{p} , [~,~], \llbracket ~, ~, ~ \rrbracket)$ is a Lie-Yamaguti algebra, where
%   \begin{align*}
%       [x, y] := x \diamond y - y \diamond x,
%   \end{align*}
%\end{prop}

%\begin{prop}
%    Let $(A , ~ \! \cdot ~ \!, \{ ~, ~, ~ \}, \{ \! \! \{ ~, ~, ~ \} \! \! \})$ be an associative-Yamaguti algebra. Then the triple $(A, ~ \! \cdot ~ \! , \langle ~, ~, ~ \rangle)$ is a preLie-Yamaguti algebra, where
%    \begin{align*}
%        \langle a, b, c \rangle := ..................
%    \end{align*}
%\end{prop}

\section{Associative-Yamaguti algebras}\label{sec3}
In this section, we introduce associative-Yamaguti algebras and provide a list of examples. In particular, any diassociative algebra yields an associative-Yamaguti algebra structure. We show that any associative-Yamaguti algebra has an enveloping associative algebra. This implies that any associative-Yamaguti algebra can be obtained from a reductive associative algebra. We end this section by showing that a suitable skew-symmetrization of an associative-Yamaguti algebra gives rise to a Lie-Yamaguti algebra structure.

\begin{defn}\label{defn-assy}
    An {\bf associative-Yamaguti algebra} (or simply an {\bf Yamaguti algebra}) is a quadruple $(A, ~\!  \cdot ~ \! , \{ ~, ~, ~\}, \{ \! \! \{ ~, ~, ~\} \! \! \})$ consisting of a vector space $A$ equipped with linear operations
    \begin{align*}
         \cdot : A \otimes A \rightarrow  A \quad \text{ and } \quad  \{ ~, ~, ~\}, \{ \! \! \{ ~, ~, ~\} \! \! \} : A \otimes A \otimes A \rightarrow A
    \end{align*}
    such that for all $a, b, c, d, e \in A$, the following set of identities are hold:
    \begin{equation}\label{ay1} \tag{AY1}
        (a \cdot b ) \cdot c -  a \cdot (b \cdot c) + \{ a, b, c \} - \{ \! \! \{ a, b, c \} \! \! \} = 0,
    \end{equation}
    \begin{equation}\label{ay2} \tag{AY2}
        \{ a \cdot b, c, d \} = \{ a, b \cdot c, d \},
    \end{equation}
    \begin{equation}\label{ay3} \tag{AY3}
        \{ a, b, c \cdot d \} = \{ a, b, c \} \cdot d,
    \end{equation}
    \begin{equation}\label{ay4} \tag{AY4}
        \{ \! \! \{ a \cdot b, c, d\} \! \! \} = a \cdot \{ \! \! \{ b, c, d \} \! \! \},
    \end{equation}
    \begin{equation}\label{ay5} \tag{AY5}
        \{ \! \! \{ a, b \cdot c, d\} \! \! \} = \{ \! \! \{ a, b, c \cdot d \} \! \! \} ,
    \end{equation}
    \begin{equation}\label{ay6} \tag{AY6}
        a \cdot \{ b, c, d \} = \{ \! \! \{ a, b, c \} \! \! \} \cdot d,
    \end{equation}
    \begin{equation}\label{ay7} \tag{AY7}
         \{ \{ a, b, c \}, d, e \} = \{ a, \{ \! \! \{ b, c, d\} \! \! \}, e \} = \{ a, b, \{ c, d , e \} \},
    \end{equation}
    \begin{equation}\label{ay8} \tag{AY8}
        \{ a, \{ b, c , d \} , e \} = \{ \{ \! \! \{ a, b, c\} \! \! \}, d, e \},
    \end{equation}
    \begin{equation}\label{ay9} \tag{AY9}
        \{ \! \! \{   \{ \! \! \{ a, b, c\} \! \! \}, d, e    \} \! \! \} = \{ \! \! \{ a, \{ b, c, d \}, e \} \! \! \} = \{ \!  \! \{ a, b, \{ \! \! \{ c, d, e\} \! \! \} \} \! \! \},
    \end{equation}
    \begin{equation}\label{ay10} \tag{AY10}
         \{ \! \! \{ a, \{ \! \! \{ b, c, d \} \! \! \}, e \} \! \! \}  = \{ \! \! \{ a, b, \{c , d, e \} \} \! \! \},
    \end{equation}
    \begin{equation}\label{ay11}\tag{AY11}
        \{ a, b, \{ \! \! \{ c, d, e \} \! \! \} \} = \{ \! \! \{ \{ a, b, c \}, d, e \} \! \! \}.
    \end{equation}
\end{defn}

\medskip

Let $(A, ~ \! \cdot ~ \! , \{ ~, ~, ~ \}, \{ \! \! \{ ~, ~ , ~ \} \! \! \})$ be an associative-Yamaguti algebra. Define linear maps $ \sigma,  \tau : A \otimes A \rightarrow \mathrm{End}(A)$ by
\begin{align*}
    \sigma (a, b) = \sigma_{a, b} := \{ a, b, - \} ~ \mathrm{~and~} ~  \tau (a, b) = \tau_{a, b} := \{ \! \! \{ -, a, b \} \! \! \}, \text{ for } a, b \in A.
\end{align*}
Then the identities (\ref{ay1})-(\ref{ay11}) can be rephrased as
\begin{align*}
    & \qquad \qquad \qquad (a \cdot b) \cdot c - a \cdot (b \cdot c) + \sigma_{a, b} (c) - \tau_{b, c} (a) = 0, \\
    & \begin{cases}
        \sigma_{a \cdot b , c} = \sigma_{a, b \cdot c}~ \!, \\
        \sigma_{a , b} (c \cdot d) = \sigma_{a, b} (c) \cdot d~ \!, \\
        \tau_{a, b} (c \cdot d) = c \cdot \tau_{a, b} (d)~ \!,\\
        \tau_{a \cdot b , c} = \tau_{a, b \cdot c}~ \!, \\
        a \cdot \sigma_{b, c} (d)= \tau_{b, c}(a) \cdot d~ \!,
    \end{cases}
    \qquad \quad
    \begin{cases}
        \sigma_{a, b} \sigma_{c, d} = \sigma_{\sigma_{a, b} (c), d} = \sigma_{a, \tau_{c, d} (b)}~ \!,\\
          \sigma_{a, \sigma_{b, c} (d)} = \sigma_{\tau_{b, c} (a) , d} ~ \!, \\
           \tau_{a, b} \tau_{c, d} = \tau_{c, \tau_{a, b} (d)} = \tau_{\sigma_{c, d} (a), b} ~ \!, \\
            \tau_{\tau_{a, b} (c) , d} = \tau_{c, \sigma_{a, b} (d)} ~ \! ,\\
            \sigma_{a, b} \tau_{c, d} = \tau_{c, d} \sigma_{a, b} ~ \!,
    \end{cases}
\end{align*}
for all $a, b, c, d \in A$. Thus, an associative-Yamaguti algebra {\em can be equivalently described} by a quadruple $(A, ~ \! \cdot ~ \!, \sigma, \tau)$ of a vector space $A$ with linear operations $\cdot: A \otimes A \rightarrow A$ and $\sigma, \tau: A \otimes A \rightarrow \mathrm{End}(A)$ that satisfy the 11 conditions given above.

Let $(A, ~ \! \cdot ~ \! , \{ ~, ~, ~ \}, \{ \! \! \{ ~, ~ , ~ \} \! \! \})$ and $(A', ~ \! \cdot' ~ \! , \{ ~, ~, ~ \}', \{ \! \! \{ ~, ~ , ~ \} \! \! \}')$ be two associative-Yamaguti algebras. A {\bf homomorphism} of associative-Yamaguti algebras from $A$ to $A'$ is a linear map $\varphi: A \rightarrow A'$ that preserves the corresponding operations. Associative-Yamaguti algebras and homomorphisms between them form a category, which we denote by ${\bf AssY}.$

\begin{exam}\label{exam-ass-assy}
    Let $(A, ~ \! \cdot ~ \!)$ be an associative algebra. Then the tuple $(A, ~ \! \cdot ~ \! , \{ ~, ~, ~ \}, \{ \! \! \{ ~, ~ , ~ \} \! \! \})$ is an associative-Yamaguti algebra, where
    \begin{align*}
        \{ a, b, c \} = \{ \! \! \{ a, b, c \} \! \! \} = (a \cdot b) \cdot c = a \cdot (b \cdot c), \text{ for all } a, b, c \in A.
    \end{align*}
\end{exam}

\begin{exam}\label{reductive}
     Let $(A, ~ \! \cdot ~ \!)$ be an associative algebra. A {\em reductive decomposition} of $(A, ~ \! \cdot ~ \!)$ is a direct sum decomposition $A= A_0 \oplus A_1$ of the underlying vector space for which $A_0 \cdot A_0 \subset A_0$, $A_0 \cdot A_1  \subset A_1$ and $A_1 \cdot A_0 \subset A_1$. In this case, $(A, ~ \! \cdot ~ \!)$ is called a {\em reductive associative algebra}. For any $a,b, c \in A_1$, we define
     \begin{align*}
         a \bullet b = \mathrm{pr}_{A_1} (a \cdot b), \quad \{a , b, c \} = (\mathrm{pr}_{A_0} (a \cdot b)) \cdot c \quad \text{ and } \quad  \{ \! \! \{ a, b, c \} \! \! \} = a \cdot (\mathrm{pr}_{A_0} (b \cdot c)).
     \end{align*}
     Here $\mathrm{pr}_{A_0}$ and $\mathrm{pr}_{A_1}$ denote the projections onto the subspaces $A_0$ and $A_1$, respectively. Then the quadruple $(A_1, ~ \! \bullet ~ \!, \{ ~, ~ , ~ \}, \{ \! \! \{ ~, ~ , ~ \} \! \! \})$ is an associative-Yamaguti algebra.
\end{exam}

\begin{exam}\label{exam-asst}
    An {\bf associative triple system} (of first kind) \cite{lister,carlsson} is a pair $(A, \{ ~, ~, ~ \})$ of a vector space $A$ with a linear operation $\{ ~, ~ , ~ \} : A \otimes A \otimes A \rightarrow A$ satisfying
    \begin{align}\label{asst-iden}
        \{ \{ a, b, c \}, d, e \} = \{ a, \{ b, c, d \}, e \} = \{a, b , \{ c, d, e \} \}, \text{ for all } a, b, c, d, e \in A.
    \end{align}
    It follows that an associative triple system $(A, \{ ~, ~, ~ \} )$ can be regarded as an associative-Yamaguti algebra $(A, ~ \! \cdot ~ \!, \{ ~, ~ , ~ \}, \{ \! \! \{ ~, ~ , ~ \} \! \! \})$ with trivial $\cdot$ and $ \{ \! \! \{ ~, ~ , ~ \} \! \! \} = \{ ~, ~, ~ \}$.
\end{exam}

\begin{exam}
    Let $(A, ~ \! \cdot ~ \!)$ be an associative algebra. For any $a \otimes a', b \otimes b', c \otimes c' \in A \otimes A$, we define
    \begin{align*}
        &(a \otimes a') \bullet (b \otimes b') = a \cdot a' \cdot b \otimes b' ~ \! + ~ \! a \otimes a' \cdot b \cdot b',\\
        \{ a \otimes a', b \otimes b', c \otimes c'  \} =~& - a \cdot a' \cdot b \cdot b' \cdot c \otimes c' ~~~ \text{ and } ~~~  \{ \! \! \{ a \otimes a', b \otimes b', c \otimes c'  \}  \! \! \} = - a \otimes a' \cdot b \cdot b' \cdot c \cdot c'.
    \end{align*}
    Then $(A \otimes A, ~ \! \bullet ~ \! , \{ ~, ~, ~ \} , \{ \! \! \{ ~, ~ , ~ \} \! \! \})$ is an associative-Yamaguti algebra.
\end{exam}

\begin{exam}
    Let $(A, ~ \! \cdot ~ \!)$ be an associative algebra and $M$ be a bimodule over it. Then the tuple $(A \oplus M, ~ \! \bullet ~ \!, \{ ~, ~ , ~ \}, \{ \! \! \{ ~, ~ , ~ \} \! \! \})$ is an associative-Yamaguti algebra, where for $(a, u), (b, v) , (c, w) \in A \oplus M$,
    \begin{align*}
        & \qquad (a, u) \bullet (b, v) = (2a \cdot b ~ \! , ~ \! a \cdot v + u \cdot b), \\
        \{ (a, u) , (b, v) , (c, w) \} =~& - (a \cdot b \cdot c ~ \! , ~ \! a \cdot b \cdot w) ~~~ \text{ and } ~~~ \{ \! \! \{ (a, u) , (b, v) , (c, w) \} \! \! \} = - (a \cdot b \cdot c ~ \! , ~ \! u \cdot b \cdot c).
    \end{align*}
    Then $(A \oplus M , ~ \! \bullet ~ \! , \{ ~, ~ , ~ \}, \{ \! \! \{ ~, ~ ,~ \} \! \! \})$ is an associative-Yamaguti algebra.
\end{exam}

In the following result, we show that a diassociative algebra canonically gives rise to an associative-Yamaguti algebra. First, recall that a {\bf diassociative algebra} \cite{loday} is a triple $(D, \dashv , \vdash)$ consisting of a vector space $D$ with two linear operations $\dashv, \vdash : D \otimes D \rightarrow D$ that make $(D, \dashv)$ and $(D, \vdash)$ both into associative algebras satisfying additionally
\begin{align}\label{dialg-iden}
    a \dashv (b \dashv c) = a \dashv (b \vdash c), \qquad (a \vdash b) \dashv c = a \vdash (b \dashv c)  \quad \text{ and } \quad (a \vdash b) \vdash c = (a \dashv b) \vdash c,
\end{align}
for all $a, b, c \in D$. Let $(D, \dashv, \vdash)$ and $(D', \dashv', \vdash')$ be two diassociative algebras. A homomorphism from $D$ to $D'$ is a linear map $\varphi : D \rightarrow D'$ satisfying $\varphi (a \dashv b) = \varphi (a) \dashv' \varphi (b)$ and $\varphi (a \vdash b) = \varphi (a) \vdash' \varphi (b)$, for all $a, b \in D$. Let $(D, \dashv, \vdash)$ be a diassociative algebra. Then it has been shown by Loday that the pair $(D, \lfloor ~, ~ \rfloor)$ is a left Leibniz algebra, where
\begin{align}
    \lfloor a, b \rfloor := a \vdash b ~-~ b \dashv a, \text{ for } a, b \in D.
\end{align}

\begin{thm}\label{diass-assy}
    Let $(D, \dashv, \vdash)$ be a diassociative algebra. Then the quadruple $(D, ~\!  \cdot ~ \! , \{ ~, ~, ~\}, \{ \! \! \{ ~, ~, ~\} \! \! \})$ is an associative-Yamaguti algebra, where for $a, b, c \in D,$
    \begin{align}
        a \cdot b :=~& a \dashv b + a \vdash b, \label{diass-assy1}\\
        \{ a, b, c \} :=~& - (a \dashv b) \vdash c = - (a \vdash b) \vdash c = - a \vdash (b \vdash c), \label{diass-assy2}\\
        \{ \! \! \{ a, b, c \} \! \! \} :=~& - a \dashv (b \vdash c) = - a \dashv (b \dashv c) = - (a \dashv b) \dashv c. \label{diass-assy3}
    \end{align} 
    Moreover, if $(D, \dashv, \vdash)$ and $(D', \dashv', \vdash')$ are two diassociative algebras and $\varphi : D \rightarrow D'$ is a homomorphism between them, then $\varphi$ is also a homomorphism between the corresponding associative-Yamaguti algebras.
\end{thm}

\begin{proof}
    For any $a, b, c \in D$, we see that
    \begin{align*}
        &(a \cdot b ) \cdot c - a \cdot ( b  \cdot c) + \{ a, b, c \} - \{ \! \! \{ a, b, c \} \! \! \}\\
        &= (a \dashv b + a \vdash b) \dashv c ~ +~  (a \dashv b + a \vdash b) \vdash c ~-~ a \dashv (b \dashv c + b \vdash c) ~-~ a \vdash (b \dashv c + b \vdash c) \\
        & \qquad \qquad ~-~ (a \dashv b) \vdash c ~+~ a \dashv (b \vdash c) ~=~ 0.
    \end{align*}
    Hence, the identity (\ref{ay1}) follows. Next, for any $a, b, c, d \in D$, we observe that
    \begin{align*}
        \{ a \cdot b, c, d \} = - \big(( a \dashv b + a \vdash b ) \dashv c \big) \vdash d =~& - \big( (a \dashv b) \dashv c + (a \vdash b)\dashv c \big) \vdash d \\
        =~& - ( a \vdash (b \vdash c) + a \vdash (b \dashv c)) \vdash d \quad (\text{by }(\ref{dialg-iden}))\\
        =~& - (a \vdash (b \cdot c)) \vdash d = \{ a, b \cdot c ,  d\},
    \end{align*}
    \begin{align*}
        \{ a, b, c \cdot d \} = - a \vdash \big(  b \vdash (c \dashv d + c \vdash d)   \big) =~& - a \vdash ( (b \vdash c) \dashv d) ~-~ a \vdash ( (b \vdash c) \vdash d) \\
        =~& - (a \vdash (b \vdash c)) \dashv d ~-~ (a \vdash (b \vdash c)) \vdash d \\
       =~& \{ a, b, c \} \dashv d  + \{ a, b, c \} \vdash d = \{ a, b, c \} \cdot d,
    \end{align*}
    \begin{align*}
        \{ \! \! \{ a \cdot b , c, d \} \! \! \} = - \big( (a \dashv b + a \vdash b) \dashv c   \big) \dashv d =~& - \big( a \dashv (b \dashv c) ~+~ a \vdash (b \dashv c)  \big) \dashv d \\
        =~& - a \dashv ( (b \dashv c) \dashv d) - a \vdash ( (b \dashv c) \dashv d) \\
        =~& a \dashv \{ \! \! \{ b, c, d \} \! \! \} ~+~ a \vdash \{ \! \! \{ b, c, d \} \! \! \} = a \cdot \{ \! \! \{ b, c, d \} \! \! \},
    \end{align*}
    \begin{align*}
         \{ \! \! \{ a, b \cdot c, d \} \! \! \} = - a \dashv \big(  (b \dashv c + b \vdash c) \dashv d  \big) =~& - a \dashv \big(  b \dashv (c \vdash d) ~+~ b \vdash (c \dashv d)  \big) \\
         =~& - a \dashv \big(   b \vdash (c \vdash d) ~+~ b \vdash (c \dashv d) \big) \\
         =~& - a \dashv (b \vdash (c \cdot d)) = \{ \! \! \{ a, b, c \cdot d \} \! \! \},
    \end{align*}
    \begin{align*}
        a \cdot \{ b, c, d \} = a \dashv \{ b, c, d \} + a \vdash \{ b, c, d \}  =~& - a \dashv ( (b \vdash c) \vdash d) - a \vdash ( (b \vdash c) \vdash d) \\
        =~& - (a \dashv (b \vdash c)) \dashv d ~-~ (a \dashv (b \vdash c)) \vdash d \quad (\text{by } (\ref{dialg-iden})) \\
        =~& \{ \! \! \{ a, b, c \} \! \! \} \dashv d ~+~ \{ \! \! \{ a, b, c \} \! \! \} \vdash d = \{ \! \! \{ a, b, c \} \! \! \} \cdot d.
    \end{align*}
    This verifies the identities (\ref{ay2})-(\ref{ay6}). Finally, for any $a, b, c, d, e \in D$,
    \begin{align*}
        \{ \{ a, b, c \}, d, e \} =  ( ( (a \vdash b) \vdash c) \vdash d) \vdash e = \begin{cases}
           = (a \vdash b) \vdash ( (c \vdash d) \vdash e) = \{ a, b , \{ c, d, e \} \}, \\\\
= ((a \vdash b) \vdash (c \vdash d)) \vdash e \\
= (( a \vdash b) \dashv (c \vdash d) ) \vdash e  \\
= \big( a \vdash (b \dashv ( c \vdash d)) \big) \vdash e \\
= - (a \vdash \{ \! \! \{ b, c, d \} \! \! \}) \vdash e = \{ a, \{ \! \! \{ b, c, d \} \! \! \} , e \},
        \end{cases}
    \end{align*}
    \begin{align*}
        \{ a, \{ b, c, d \} , e \} = \big( a \vdash (( b \vdash c) \vdash d)    \big) \vdash e =~& \big( (a \vdash  (b \vdash c)) \vdash d   \big) \vdash c \\
        =~& \big( (a \dashv (b \vdash c)) \vdash d   \big) \vdash e \\
        =~& - ( \{ \! \! \{ a, b, c \} \! \! \} \vdash d) \vdash e = \{ \{ \! \! \{ a, b, c \} \! \! \}, d, e \}, 
    \end{align*}
    \begin{align*}
        \{ \! \! \{     \{ \! \! \{ a, b, c \} \! \! \}, d, e \} \! \! \} = (a \dashv (b \dashv c)) \dashv (d \dashv e) = \begin{cases}
            = a \dashv \big(b \dashv (c \dashv (d \dashv e)) \big) = \{ \! \! \{ a, b , \{ \! \! \{ c, d, e \} \! \! \} \} \! \! \}, \\\\
            = a \dashv (  (b \dashv c) \dashv (d \dashv e) ) \\
            = a \dashv ( (b \dashv c) \vdash (d \dashv e)) \\
            = a \dashv \big(  (( b \dashv c) \vdash d) \dashv e \big) \\
            = - a \dashv (\{ b, c, d \} \dashv e) = \{ \! \! \{ a, \{ b, c, d \}, e \} \! \! \},
        \end{cases}
    \end{align*}
    \begin{align*}
         \{ \! \! \{ a, \{ \! \! \{ b, c, d \} \! \! \}, e \} \! \! \} = a \dashv \big( (b \dashv (c \dashv d)) \dashv e    \big) =~& a \dashv (    b \dashv (( c \dashv d ) \dashv e) ) \\
         =~& a \dashv ( b \dashv (( c \dashv d) \vdash e) ) \\
         =~& - a \dashv (b \dashv \{ c, d , e \} ) = \{ \! \! \{ a, b, \{ c, d, e \} \} \! \! \}
    \end{align*}
    \begin{align*}
         \{ a, b, \{ \! \! \{ c, d, e \} \! \! \} \} = (a \dashv b) \vdash ((c \dashv d) \dashv ) =~& (( a \dashv b) \vdash (c \dashv d) ) \dashv e \\
         =~& \big(  (( a \dashv b) \vdash c) \dashv d    \big) \dashv e \\
         =~& - ( \{ a, b, c \} \dashv d) \dashv e = \{ \! \! \{ \{ a, b, c \} , d , e \} \! \! \}.
    \end{align*}
    Hence the identities (\ref{ay7})-(\ref{ay11}) also hold. This proves the first part. The verification of the second part is simple.
\end{proof}

As a consequence of the above result, we obtain a functor $\mathcal{F} : {\bf Diass} \rightarrow {\bf AssY}$ from the category ${\bf Diass}$ of diassociative algebras to the category ${\bf AssY}$ of associative-Yamaguti algebras.

\begin{exam}
    Let $(A, ~ \! \cdot ~ \! )$ be an associative algebra. A linear map $P : A \rightarrow A$ is said to be an {\em averaging operator} on $(A, ~ \! \cdot ~ \!)$ if
    \begin{align*}
         P(a) \cdot P(b) = P (P(a) \cdot b) = P (a \cdot P(b)), \text{ for all } a, b \in A.
    \end{align*}
    In this case, the triple $(A, \dashv_P, \vdash_P)$ is a diassociative algebra, where $a \dashv_P b = a \cdot P(b)$ and $a \vdash_P b = P(a) \cdot b$, for all $a, b \in A$. As a result, the tuple $(A, ~ \! \bullet ~ \! , \{ ~, ~ , ~ \}, \{ \! \! \{ ~, ~, ~ \} \! \! \})$ is an associative-Yamaguti algebra, where
    \begin{align*}
        a \bullet b = P(a) \cdot b + a \cdot P(b), \quad \{ a, b, c \} = - P(a) \cdot P(b) \cdot c ~~~ \text{ and } ~~~ \{ \! \! \{ a, b, c \} \! \! \} = - a \cdot P(b) \cdot P(c), ~ \text{ for } a, b, c \in A. 
    \end{align*}
\end{exam}

\medskip

In the following, we aim to show that an associative-Yamaguti algebra always admits an enveloping associative algebra. To show this, we begin with the following result.

\begin{prop}
    Let $(A, ~ \! \cdot ~ \! , \{ ~, ~, ~ \}, \{ \! \! \{ ~, ~, ~ \} \! \! \})$ be an associative-Yamaguti algebra. Suppose $(B, *)$ is any associative algebra with linear maps $\triangleright : B \otimes A \rightarrow A$ and $\triangleleft : A \otimes B \rightarrow A$ that make $A$ into a bimodule over the associative algebra $(B, *)$, satisfying additionally
    \begin{align}\label{b-algebra}
        \xi \triangleright (a \cdot b) = ( \xi \triangleright a ) \cdot b,  \quad (a \cdot b ) \triangleleft \xi = a \cdot (b \triangleleft \xi), \quad a \cdot (\xi \triangleright b ) = (a \triangleleft \xi) \cdot b,
    \end{align}
    for all $a, b \in A$ and $\xi \in B$. Also, let $\Delta : A \otimes A \rightarrow B$ be a linear map that satisfies
    \begin{equation}\label{delta-cond1}
        (\Delta (a, b)) \triangleright c= \{ a, b, c \}, \quad a \triangleleft (\Delta (b, c)) = \{ \! \! \{ a, b, c \} \! \! \},
    \end{equation}
    \begin{equation}\label{delta-cond2}
        \xi * \Delta (a, b)= \Delta (\xi \triangleright a, b), \quad \Delta (a, b) * \xi = \Delta (a, b \triangleleft \xi), \quad \Delta (a, \xi \triangleright b ) = \Delta (a \triangleleft \xi, b),
    \end{equation}
    \begin{equation}\label{delta-cond3}
        \Delta (a \cdot b , c) = \Delta (a, b \cdot c),
    \end{equation}
    for all $a, b , c \in A$ and $\xi \in B$. Then $( B \oplus A, ~ \! \oast ~ \! )$ is a reductive associative algebra, where
    \begin{align}\label{env-product}
        (\xi, a) \oast (\eta,b) = \big(    \xi * \eta +\Delta (a, b) ~ \! , ~ \! \xi \triangleright b + a \triangleleft \eta + a \cdot b \big),
    \end{align}
    for $(\xi, a) , (\eta,b) \in B \oplus A$. Moreover, the induced associative-Yamaguti algebra structure on $A$ coincides with the prescribed one.
\end{prop}

\begin{proof}
    For any $(\xi, a) , (\eta , b) , (\zeta, c) \in B \oplus A$, we observe that
    \begin{align}\label{env1}
        &((\xi, a) \oast (\eta, b) ) \oast (\zeta , c) \nonumber \\
        &= \big(    \xi * \eta +\Delta (a, b) ~ \! , ~ \! \xi \triangleright b + a \triangleleft \eta + a \cdot b \big) \oast (\zeta, c) \nonumber \\
        &= \big( (\xi \ast \eta) \ast \zeta   + \Delta (a, b) * \zeta + \Delta (\xi \triangleright b , c) + \Delta (a \triangleleft \eta, c) + \Delta (a \cdot b, c) ~ \! , ~ \! (\xi * \eta) \triangleright c + (\Delta (a, b)) \triangleright c  \nonumber  \\
       & \qquad \qquad + (\xi \triangleright b) \triangleleft \zeta + (a \triangleleft \eta) \triangleleft \zeta + (a \cdot b) \triangleleft \zeta + (\xi \triangleright b) \cdot c + (a \triangleleft \eta) \cdot c + (a \cdot b) \cdot c \big). 
    \end{align}
    On the other hand,
    \begin{align}\label{env2}
             &(\xi, a) \oast ((\eta, b)  \oast (\zeta , c) ) \nonumber \\
             &= (\xi, a) \oast \big(  \eta * \zeta + \Delta (b, c) ~ \!, ~ \! \eta \triangleright c + b  \triangleleft \zeta + b \cdot  c \big) \nonumber \\
             &= \big( \xi * (\eta * \zeta) + \xi * \Delta (b, c) + \Delta (a, \eta \triangleright c) + \Delta (a, b \triangleleft \zeta) + \Delta (a, b \cdot c) ~ \!, ~ \! \xi \triangleright (\eta \triangleright c) + \xi \triangleright (b \triangleleft \zeta)   \nonumber \\
             & \qquad \qquad + \xi \triangleright (b \cdot c) + a \triangleleft (\eta * \zeta) + a \triangleleft ( \Delta (b, c)) + a \cdot (\eta \triangleright c) + a \cdot (b \triangleleft \zeta) + a \cdot (b \cdot c) \big). 
    \end{align}
    By applying the assumptions provided in the statement, one can observe that the expressions in (\ref{env1}) and (\ref{env2}) are identical. Hence $(B \oplus A, \oast)$ is an associative algebra. Further, it follows from the expression (\ref{env-product}) that 
    \begin{align*}
        (\xi, 0) \oast (\eta, 0) \in B \oplus \{ 0 \}, \quad (\xi, 0) \oast (0, a) \in \{ 0 \} \oplus A ~~~~ \text{ and } ~~~~ (0, a) \oast (\xi, 0) \in \{ 0 \} \oplus A,
    \end{align*}
    for any $\xi, \eta \in B$ and $a \in A$. This shows that $(B \oplus A, \oast)$ is a reductive algebra. Finally, the last part is easy to verify.
\end{proof}

Let $(A, ~ \! \cdot ~ \!, \{ ~, ~, ~ \}, \{ \! \! \{ ~, ~, ~ \} \! \! \})$ be an associative-Yamaguti algebra. Let $(B, *)$ be any associative algebra with linear maps $\triangleright : B \otimes A \rightarrow A$, $\triangleleft : A \otimes B \rightarrow A$ that make $A$ into a bimodule over the associative algebra $(B, *)$, satisfying additionally (\ref{b-algebra}). Also, let $\Delta : A \otimes A \rightarrow B$ be a linear map satisfying the conditions given in (\ref{delta-cond1})-(\ref{delta-cond3}). Then the associative algebra $(B \oplus A , \oast)$ constructed in the above proposition is called an {\bf enveloping associative algebra} of the associative-Yamaguti algebra $(A, ~ \! \cdot ~ \!, \{ ~, ~, ~ \}, \{ \! \! \{ ~, ~, ~ \} \! \! \})$.

\medskip

The next result shows that any associative-Yamaguti algebra admits an enveloping associative algebra. The proof of this result is straightforward and follows from the axioms of an associative-Yamaguti algebra. In particular, this shows that any associative-Yamaguti algebra is always induced by a reductive associative algebra in the sense of Example \ref{reductive}.

\begin{thm}\label{thm-env-ass}
    Let $(A, ~ \! \cdot ~ \!, \{ ~,~,~\}, \{ \! \! \{ ~, ~, ~ \} \! \! \})$ be an associative-Yamaguti algebra. Suppose $\mathcal{M} (A)$ is the vector subspace of $\mathrm{End}(A) \oplus \mathrm{End}(A)$ spanned by the set $\{  (\sigma_{a, b} , \tau_{a, b}) ~ \! | ~ \! a, b \in A    \}$.
\begin{itemize}
    \item[(i)] Then $( \mathcal{M}(A), ~ \! * ~ \!)$ is an associative algebra, where
    \begin{align*}
        (\sigma_{a, b} , \tau_{a, b}) * (\sigma_{c, d} , \tau_{c, d}) = ( \sigma_{ \{ a, b, c \} ,  d }, \tau_{ \{ a, b, c \} ,  d } ), \text{ for }  (\sigma_{a, b} , \tau_{a, b}) , (\sigma_{c, d} , \tau_{c, d}) \in \mathcal{M}(A).
    \end{align*}
    \item[(ii)] We define linear maps $\triangleright : \mathcal{M} (A) \otimes A \rightarrow A$ and $\triangleleft : A \otimes \mathcal{M} (A) \rightarrow A$ by
    \begin{align*}
        (\sigma_{a, b} , \tau_{a, b})  \triangleright c = \sigma_{a, b} (c) = \{ a, b, c \} ~~~~ \text{ and } ~~~~ c \triangleleft (\sigma_{a, b} , \tau_{a, b}) = \tau_{a, b } (c) = \{ \! \! \{ c, a, b \} \! \! \},
    \end{align*}
    for $(\sigma_{a, b} , \tau_{a, b}) \in \mathcal{M} (A)$ and $c \in A$. With the above action maps, $A$ is a bimodule over the associative algebra $(\mathcal{M} (A), ~ \! * ~ \!)$ and the identities in (\ref{b-algebra}) hold.
    \item[(iii)] Also, we define a map $\Delta : A \otimes A \rightarrow \mathcal{M} (A)$ by $\Delta (a, b) = (\sigma_{a, b} , \tau_{a, b})$, for $a, b \in A$. Then $\Delta$ satisfy the identities given in (\ref{delta-cond1})-(\ref{delta-cond3}). 
\end{itemize}
As a consequence, $(\mathcal{M} (A) \oplus A, ~ \! \oast ~ \!)$ is a reductive associtaive algebra, where
\begin{align*}
    (\sigma_{a, b} , \tau_{a, b}, x) \oast (\sigma_{c, d} , \tau_{c, d}, y) = \big( \sigma_{ \{ a, b, c \} ,  d } + \sigma_{x, y} ~ \! , ~ \! \tau_{ \{ a, b, c \} ,  d } + \tau_{x, y} ~ \!, ~ \! \{ a, b, y \} + \{ \! \! \{ x, c, d \} \! \! \} + x \cdot y \big),
\end{align*}
for $(\sigma_{a, b} , \tau_{a, b}, x) , (\sigma_{c, d} , \tau_{c, d}, y) \in \mathcal{M}(A) \oplus A$. Further, the induced associative-Yamaguti algebra structure on $A$ coincides with the prescribed one.
\end{thm}

In the following result, we show that a suitable skew-symmetrization of an associative-Yamaguti algebra gives rise to a Lie-Yamaguti algebra structure. More precisely, we have the following.

\begin{thm}\label{assy-liey}
    Let $(A, ~\!  \cdot ~ \! , \{ ~, ~, ~\}, \{ \! \! \{ ~, ~, ~\} \! \! \})$ be an associative-Yamaguti algebra. Then $(A, [~,~], \llbracket ~, ~, ~ \rrbracket)$ is a Lie-Yamaguti algebra, where
\begin{align*}
    [a, b ] = a \cdot b  - b \cdot a \quad  \text{ and }  \quad \llbracket a, b, c \rrbracket = \{ a, b, c \}- \{ b, a, c \} - \{ \! \! \{ c, a, b \} \! \! \} + \{ \! \! \{ c, b, a \} \! \! \}, \text{ for } a, b, c \in A.
\end{align*}
Moreover, if $(A, ~\!  \cdot ~ \! , \{ ~, ~, ~\}, \{ \! \! \{ ~, ~, ~\} \! \! \})$ and $(A', ~\!  \cdot' ~ \! , \{ ~, ~, ~\}', \{ \! \! \{ ~, ~, ~\} \! \! \}')$ are two associative-Yamaguti algebras and $\varphi : A \rightarrow A'$ is a homomorphism between them, then $\varphi$ is also a homomorphism between the corresponding Lie-Yamaguti algebras.
\end{thm}

\begin{proof}
    It is easy to see that the bracket $[~,~]$ is skew-symmetric and the operation $\llbracket ~, ~, ~ \rrbracket$ is skew-symmetric on the first two inputs. Let $a,b, c \in A$ be arbitrary. Then we have
    \begin{align*}
       & [[a, b ], c] + [[b, c], a] + [[c, a], b] + \llbracket a, b, c \rrbracket + \llbracket b, c, a \rrbracket + \llbracket c, a, b \rrbracket \\
       & = (a \cdot b - b \cdot a) \cdot c - c \cdot (a \cdot b - b \cdot a) + (b \cdot c - c \cdot b) \cdot a - a \cdot (b \cdot c - c \cdot b) \\ 
       & \quad + (c \cdot a - a \cdot c) \cdot b - b \cdot (c \cdot a - a \cdot c) 
        + \{ a, b, c \} - \{ b, a, c \} - \{ \! \! \{ c, a, b \} \! \! \} + \{ \! \! \{ c, b, a    \} \! \! \} \\
        & \quad + \{ b, c, a \} - \{ c, b, a \} - \{ \! \! \{ a, b, c \} \! \! \} + \{ \! \! \{     a, c, b \} \! \! \} + \{ c, a, b \} - \{ a, c, b \} - \{ \! \! \{ b, c, a    \} \! \! \} + \{ \! \! \{ b, a, c    \} \! \! \} \\
        &= \big(   (a \cdot b) \cdot c - a \cdot (b \cdot c) + \{ a, b, c \} - \{ \! \! \{    a, b, c \} \! \! \}  \big) + \big(  (b \cdot c) \cdot a - b \cdot (c \cdot a) + \{ b, c, a \} - \{ \! \! \{ b, c, a    \} \! \! \}   \big) \\
        & \quad + \big(  (c \cdot a) \cdot b - c \cdot (a \cdot b) + \{ c, a, b \} - \{ \! \! \{     c, a, b \} \! \! \}  \big) - \big(  (b \cdot a) \cdot c - b  \cdot (a \cdot c) + \{ b, a, c \} - \{ \! \! \{  b, a, c   \} \! \! \}  \big) \\
       & \quad - \big(  (a \cdot c) \cdot b - a \cdot (c \cdot b) + \{ a, c, b \} - \{ \! \! \{     a, c, b \} \! \! \}  \big) - \big(  (c \cdot b) \cdot a - c \cdot (b \cdot a) + \{ c, b, a \} - \{ \! \! \{  c, b, a    \} \! \! \}     \big) \\
    %   &\qquad \qquad \qquad  \qquad \qquad \qquad  \qquad \qquad \qquad  \qquad \qquad \qquad  (\text{after rearrangement}) \\
        &= 0 \qquad (\text{by } (\text{\ref{ay1}}))
    \end{align*}
    which proves the identity (\ref{ly1}). For any $a, b, c, d \in A$, we also have
    \begin{align*}
       & \llbracket [a, b], c, d \rrbracket + \llbracket [b, c],a, d \rrbracket + \llbracket [c, a], b, d \rrbracket \\
        &= \{ a \cdot b - b \cdot a , c, d \} - \{ c, a\cdot b - b \cdot a, d \} - \{ \! \! \{  d, a \cdot b - b \cdot a, c   \} \! \! \} + \{ \! \! \{  d, c, a \cdot b - b \cdot a   \} \! \! \} \\
       & \quad + \{ b \cdot c - c \cdot b, a, d \} - \{ a, b \cdot c - c \cdot b , d \} - \{ \! \! \{  d, b \cdot c - c \cdot b, a   \} \! \! \} + \{ \! \! \{ d, a, b \cdot c - c \cdot b    \} \! \! \} \\
       & \quad + \{ c \cdot a - a \cdot c, b, d \} - \{ b, c \cdot a - a \cdot c, d \} - \{ \! \! \{ d, c \cdot a - a \cdot c, b    \} \! \! \} + \{ \! \! \{  d, b, c \cdot a- a \cdot c   \} \! \! \} \\
       & = \big(  \{ a \cdot b , c, d \} - \{ a, b \cdot c, d \}  \big) - \big(  \{ b \cdot a, c, d \} - \{ b, a \cdot c, d \}  \big) + \big(  \{ b \cdot c, a, d \} - \{ b, c \cdot a, d \}  \big) \\
       & \quad - \big(  \{ c \cdot b, a, d \} - \{ c, b \cdot a , d \} \big) + \big(  \{ c \cdot a, b, d \} - \{ c, a \cdot b , d \} \big) - \big(   \{ a \cdot c, b, d \} - \{ a, c \cdot b, d \}   \big) \\
       & \quad - \big(    \{ \! \! \{ d, a \cdot b, c    \} \! \! \} - \{ \! \! \{ d, a, b \cdot c    \} \! \! \}   \big) + \big(   \{ \! \! \{  d, b \cdot a, c   \} \! \! \} - \{ \! \! \{ d, b, a \cdot c    \} \! \! \}  \big) - \big(  \{ \! \! \{ d, b \cdot c, a    \} \! \! \} - \{ \! \! \{ d, b, c \cdot a    \} \! \! \}  \big) \\
       & \quad + \big(   \{ \! \! \{ d, c \cdot b, a    \} \! \! \}   - \{ \! \! \{  d, c, b \cdot a   \} \! \! \} \big) - \big(   \{ \! \! \{ d, c \cdot a, b    \} \! \! \} - \{ \! \! \{  d, c, a \cdot b   \} \! \! \}  \big) + \big(   \{ \! \! \{ d, a \cdot c, b    \} \! \! \}  - \{ \! \! \{ d , a, c \cdot b    \} \! \! \} \big) \\
    %   & \qquad \qquad \qquad  \qquad \qquad \qquad  \qquad \qquad \qquad  \qquad \qquad \qquad  (\text{after rearrangement}) \\
       & = 0 \qquad (\text{by } (\text{\ref{ay2}}) \text{ and } (\text{\ref{ay5}})).
    \end{align*}
    Hence, the identity (\ref{ly2}) also follows. Moreover, we see that
    \begin{align*}
        &\llbracket a, b, [c, d ] \rrbracket - [ \llbracket a, b, c \rrbracket , d] - [c, \llbracket a, b, d \rrbracket ] \\
        &= \{ a, b, c \cdot d - d \cdot c \} - \{ b, a, c \cdot d - d \cdot c \} - \{ \! \! \{ c \cdot d - d \cdot c, a, b   \} \! \! \} + \{ \! \! \{ c \cdot d - d \cdot c, b, a   \} \! \! \} \\
        & \quad- \{ a, b, c \} \cdot d + \{ b, a, c \} \cdot d + \{ \! \! \{  c, a, b  \} \! \! \} \cdot d - \{ \! \! \{  c, b, a  \} \! \! \} \cdot d + d \cdot \{ a, b, c \} - d \cdot \{ b, a, c \} \\
        & \quad - d \cdot \{ \! \! \{ c, a, b   \} \! \! \} + d \cdot \{ \! \! \{  c, b, a  \} \! \! \} - c \cdot \{ a, b, d \} + c \cdot \{ b, a, d \} + c \cdot \{ \! \! \{ d, a, b   \} \! \! \} - c \cdot \{ \! \! \{ d, b, a   \} \! \! \} \\
       & \quad + \{a, b, d \} \cdot c - \{ b, a, d \} \cdot c - \{ \! \! \{ d, a, b \} \! \! \} \cdot c + \{ \! \! \{ d, b, a   \} \! \! \} \cdot c \\
       & = \big(  \{ a, b, c \cdot d \} - \{ a, b, c \} \cdot d    \big) - \big( \{ a, b, d \cdot c \} - \{ a, b, d \} \cdot c    \big) - \big(   \{ b, a, c \cdot d \} - \{ b, a, c \} \cdot d  \big) \\
       & \quad + \big(  \{ b, a, d \cdot c \} - \{ b, a, d \} \cdot c  \big) - \big(   \{ \! \! \{   c \cdot d, a, b \} \! \! \} - c \cdot \{ \! \! \{ d, a, b   \} \! \! \} \big) + \big(   \{ \! \! \{ d \cdot c, a, b   \} \! \! \}  - d \cdot \{ \! \! \{ c, a, b   \} \! \! \}  \big) \\
        & \quad + \big(    \{ \! \! \{  c \cdot d, b, a  \} \! \! \}  - c \cdot \{ \! \! \{ d, b, a   \} \! \! \}  \big) - \big(  \{ \! \! \{ d \cdot c, b, a   \} \! \! \} - d \cdot \{ \! \! \{  c, b, a  \} \! \! \}  \big) - \big(  c \cdot \{  a, b, d  \}  - \{ \! \! \{ c, a, b   \} \! \! \} \cdot d  \big) \\
       & \quad + \big(  c \cdot \{ b, a, d \} - \{ \! \! \{ c, b, a   \} \! \! \} \cdot d   \big) + \big( d \cdot \{ a, b, c   \} - \{ \! \! \{ d, a, b   \} \! \! \} \cdot c \big) - \big(   d \cdot \{ b, a, c \} - \{ \! \! \{ d, b, a   \} \! \! \} \cdot c \big) \\
     %   & \qquad \qquad \qquad  \qquad \qquad \qquad  \qquad \qquad \qquad  \qquad \qquad \qquad (\text{after rearrangement}) \\
        &= 0 \qquad (\text{by } (\text{\ref{ay3}}), (\text{\ref{ay4}}) \text{ and } (\text{\ref{ay6}})),
    \end{align*}
    that verifies (\ref{ly3}). Finally, for any $a, b, c, d, e \in A$, we see that
    \begin{align*}
        &\llbracket a, b, \llbracket c, d, e \rrbracket \rrbracket - \llbracket \llbracket a, b, c \rrbracket , d, e \rrbracket - \llbracket c, \llbracket a, b, d \rrbracket , e \rrbracket - \llbracket c, d , \llbracket a, b, e \rrbracket \rrbracket \\
        &= \big\{ a, b, \{ c, d, e \} - \{ d, c, e \} - \{ \! \! \{ e, c, d \} \! \! \} + \{ \! \! \{ e, d, c    \} \! \! \} \big\} - \big\{ b, a , \{ c, d, e \} - \{ d, c, e \} - \{ \! \! \{ e, c, d \} \! \! \} + \{ \! \! \{ e, d, c    \} \! \! \} \big\} \\
        & \quad - \{ \! \! \{ \{ c, d, e \} - \{ d, c, e \} - \{ \! \! \{ e, c, d \} \! \! \} + \{ \! \! \! \{ e, d, c    \} \! \! \}, a, b \} \! \! \}  + \{ \! \! \{    \{ c, d, e \} - \{ d, c, e \} - \{ \! \! \{ e, c, d \} \! \! \} + \{ \! \! \{ e, d, c    \} \! \! \}, b, a \} \! \! \} \\
        & \quad - \big\{   \{a, b, c \} - \{ b, a, c \}- \{ \! \! \{ c, a, b \} \! \! \}+ \{ \! \! \{ c, b, a \} \! \! \}, d, e   \big\} + \big\{  d,  \{a, b, c \} - \{ b, a, c \}- \{ \! \! \{ c, a, b \} \! \! \}+ \{ \! \! \{ c, b, a \} \! \! \}, e \big\} \\
        & \quad + \{ \! \! \{ e,  \{a, b, c \} - \{ b, a, c \}- \{ \! \! \{ c, a, b \} \! \! \}+ \{ \! \! \{ c, b, a \} \! \! \}, d \} \! \! \} - \{ \! \! \{ e, d,  \{a, b, c \} - \{ b, a, c \}- \{ \! \! \{ c, a, b \} \! \! \}+ \{ \! \! \{ c, b, a \} \! \! \} \} \! \! \} \\
        & \quad - \big\{       c, \{ a, b, d \} - \{ b, a, d \} - \{ \! \! \{ d, a, b \} \! \! \} + \{ \! \! \{ d, b, a \} \! \! \}, e \big\} + \big\{    \{ a, b, d \} - \{ b, a, d \} - \{ \! \! \{ d, a, b \} \! \! \} + \{ \! \! \{ d, b, a \} \! \! \}, c, e  \big\} \\
        & \quad + \{ \! \! \{ e, c,  \{ a, b, d \} - \{ b, a, d \} - \{ \! \! \{ d, a, b \} \! \! \} + \{ \! \! \{ d, b, a \} \! \! \}  \} \! \! \} - \{ \! \! \{ e,  \{ a, b, d \} - \{ b, a, d \} - \{ \! \! \{ d, a, b \} \! \! \} + \{ \! \! \{ d, b, a \} \! \! \}, c \} \! \! \} \\
        & \quad - \big\{     c, d, \{ a, b, e \} - \{ b, a, e \} - \{ \! \! \{ e, a, b \} \! \! \} + \{ \! \! \{ e, b, a \} \! \! \} \big\} + \big\{  d, c,  \{ a, b, e \} - \{ b, a, e \} - \{ \! \! \{ e, a, b \} \! \! \} + \{ \! \! \{ e, b, a \} \! \! \}  \big\} \\
        & \quad + \{ \! \! \{ \{ a, b, e \} - \{ b, a, e \} - \{ \! \! \{ e, a, b \} \! \! \} + \{ \! \! \{ e, b, a \} \! \! \}, c, d \} \! \! \} - \{ \! \! \{ \{ a, b, e \} - \{ b, a, e \} - \{ \! \! \{ e, a, b \} \! \! \} + \{ \! \! \{ e, b, a \} \! \! \}, d, c \} \! \! \} \\
        &=  \big(  \{ a, b, \{ c, d, e \} \} - \{ \{ a, b, c \}, d, e \}  \big) - \big(   \{ a, b , \{ d, c, e \} \} - \{ \{ a, b, d \} , c, e \}   \big) \\ 
        & \qquad  - \big(   \{ a, b,   \{ \! \! \{  e,c, d    \} \! \! \} \} -  \{ \! \! \{  \{ a, b, e \}, c, d    \} \! \! \}    \big)  + \big(   \{ a, b,   \{ \! \! \{  e, d , c   \} \! \! \} \} -  \{ \! \! \{  \{ a, b, e \}, d, c    \} \! \! \}    \big) \\
        & \qquad    - \big(  \{ b, a , \{ c, d, e \} \} - \{ \{ b, a, c \}, d, e \}   \big) +  \big(  \{ b, a , \{  d, c, e \} \} - \{ \{ b, a, d \}, c, e \}     \big) \\
        & \qquad   + \big(  \{ b, a, \{ \! \! \{ e, c, d \} \! \! \} \} - \{ \! \! \{ \{ b, a, e \}, c, d \} \! \! \}    \big) - \big(  \{ b, a, \{ \! \! \{ e,  d, c \} \! \! \} \} - \{ \! \! \{ \{ b, a, e \},  d, c \} \! \! \}    \big) \\
        & \qquad   - \big(   \{ \! \! \{  \{ c, d, e \}, a, b   \} \! \! \}  - \{ c, d ,  \{ \! \! \{   e, a, b  \} \! \! \} \}  \big) + \big(   \{ \! \! \{ \{ d, c, e \}, a, b    \} \! \! \}  - \{ d, c,  \{ \! \! \{ e, a, b    \} \! \! \} \}    \big) \\
        & \qquad   + \big(    \{ \! \! \{      \{ \! \! \{  e, c, d   \} \! \! \} , a, b   \} \! \! \}   -  \{ \! \! \{  e, c,  \{ \! \! \{  d, a, b   \} \! \! \}     \} \! \! \}   \big) - \big(   \{ \! \! \{    \{ \! \! \{  e, d, c   \} \! \! \}, a, b    \} \! \! \} -  \{ \! \! \{ e, d,  \{ \! \! \{ c, a, b    \} \! \! \}     \} \! \! \}   \big) \\
        & \qquad  + \big(    \{ \! \! \{  \{c, d, e \}, b, a   \} \! \! \}  - \{ c, d,   \{ \! \! \{e, b, a   \} \! \! \}   \}  \big)  - \big(   \{ \! \! \{  \{ d, c, e \}, b, a   \} \! \! \}  - \{ d, c,  \{ \! \! \{  e, b, a   \} \! \! \} \}  \big) \\
        & \qquad   - \big(   \{ \! \! \{    \{ \! \! \{   e, c, d  \} \! \! \} , b, a    \} \! \! \} -  \{ \! \! \{  e, c,  \{ \! \! \{  d, b, a   \} \! \! \}    \} \! \! \}     \big) + \big(    \{ \! \! \{    \{ \! \! \{  e, d, c   \} \! \! \} , b, a   \} \! \! \}    -  \{ \! \! \{  e, d,  \{ \! \! \{  c, b, a   \} \! \! \}    \} \! \! \} \big) \\
        & \qquad   + \big(  \{    \{ \! \! \{  c, a, b   \} \! \! \} , d, e \} - \{ c, \{ a, b, d \}, e \}   \big) -  \big(  \{  \{ \! \! \{  c, b, a   \} \! \! \}, d, e \} - \{ c,  \{  b, a, d    \} , e \}   \big) \\
        & \qquad   + \big( \{ d, \{ a, b, c \}, e \} - \{  \{ \! \! \{  d, a, b   \} \! \! \} , c, e \}   \big) - \big( \{ d, \{ b, a, c \}, e \} - \{   \{ \! \! \{  d, b, a   \} \! \! \} , c, e \} \big) \\
        & \qquad   - \big(   \{ d,  \{ \! \! \{ c, a, b    \} \! \! \}, e \} - \{ d, c, \{ a, b, e \} \}    \big) + \big(  \{ d,  \{ \! \! \{  c, b, a   \} \! \! \} , e \} - \{ d, c, \{ b, a, e \}  \}    \big)  \\
        & \qquad   + \big(   \{ \! \! \{ e, \{ a, b, c \}, d    \} \! \! \}  -  \{ \! \! \{   \{ \! \! \{  e, a, b   \} \! \! \} , c, d   \} \! \! \}    \big) - \big(   \{ \! \! \{  e, \{ b, a, c \}, d   \} \! \! \}   -  \{ \! \! \{   \{ \! \! \{  e, b, a   \} \! \! \}    , c, d \} \! \! \}    \big) \\
        & \qquad   - \big(  \{ \! \! \{  e,  \{ \! \! \{  c, a, b     \} \! \! \} , d     \} \! \! \} -  \{ \! \! \{   e, c , \{ a, b, d \}    \} \! \! \}    \big) + \big(   \{ \! \! \{    e,  \{ \! \! \{ c, b, a      \} \! \! \} , d   \} \! \! \} -  \{ \! \! \{   e, c, \{ b, a, d \}    \} \! \! \}   \big) \\
       &  \qquad  - \big(  \{ \! \! \{  e, d, \{ a, b, c \}     \} \! \! \}   -   \{ \! \! \{  e,  \{ \! \! \{   d, a, b    \} \! \! \} , c     \} \! \! \}    \big) + \big(    \{ \! \! \{  e, d,  \{  b, a, c     \}      \} \! \! \}  -  \{ \! \! \{ e,  \{ \! \! \{   d, b, a    \} \! \! \} , c      \} \! \! \}     \big) \\
      &  \qquad+ \big( \{ c,  \{ \! \! \{ d, a, b      \} \! \! \} , e \} - \{ c, d, \{ a, b, e \} \}    \big) - \big(  \{ c,  \{ \! \! \{  d, b, a     \} \! \! \}, e \} - \{ c, d, \{ b, a, e \} \}   \big) \\
       &  \qquad  - \big(  \{ \! \! \{ e, \{ a, b, d \}, c      \} \! \! \}   -  \{ \! \! \{      \{ \! \! \{  e, a, b     \} \! \! \} , d, c    \} \! \! \}   \big) + \big(   \{ \! \! \{  e, \{ b, a, d \}, c     \} \! \! \}  -  \{ \! \! \{    \{ \! \! \{  e, b, a     \} \! \! \} , d, c     \} \! \! \}    \big) \\
      %  & \qquad \qquad \qquad \qquad  \qquad \qquad \qquad \qquad (\text{after rearrangement}) \\
       & = 0 \quad (\text{by using } (\text{\ref{ay7}}), (\text{\ref{ay8}}), (\text{\ref{ay9}}), (\text{\ref{ay10}}) \text{ and } (\text{\ref{ay11}})).
    \end{align*}
    This verifies the identity (\ref{ly4}). Hence, $(A, [~,~], \llbracket ~, ~ , ~ \rrbracket)$ is a Lie-Yamaguti algebra.

    For the second part, we observe that
    \begin{align*}
        \varphi ([a, b]) = \varphi (a \cdot b - b \cdot a) = \varphi (a) \cdot' \varphi (b) - \varphi (b) \cdot' \varphi (a) = [\varphi (a), \varphi (b)]',
    \end{align*}
    \begin{align*}
        \varphi (\llbracket a, b, c \rrbracket) =~& \varphi ( \{ a, b, c \} - \{ b, a, c \} - \{ \! \! \{ c, a, b \} \! \! \} + \{ \! \! \{ c, b, a \} \! \! \} ) \\
        =~& \{ \varphi (a), \varphi( b), \varphi(c) \}' - \{\varphi( b), \varphi( a), \varphi(c) \}' - \{ \! \! \{ \varphi(c), \varphi(a), \varphi(b) \} \! \! \}' + \{ \! \! \{ \varphi (c),\varphi( b), \varphi(a) \} \! \! \}' \\
        =~& \llbracket \varphi (a), \varphi (b), \varphi (c) \rrbracket',
    \end{align*}
    for all $a, b, c \in A$. This proves the desired result.
\end{proof}

The above result shows the existence of a functor $\mathcal{S} : {\bf AssY} \rightarrow {\bf LieY}$ from the category of associative-Yamaguti algebras to the category of Lie-Yamaguti algebras. 

\begin{remark}
    \begin{itemize}
        \item[(i)] The above construction of a Lie-Yamaguti algebra from an associative-Yamaguti algebra is compatible with the well-known `skew-summetrization' construction of a Lie algebra from an associative algebra. More precisely, the following diagram is commutative:
\begin{align}\label{ass-diagram}
   \xymatrix{
     \mathrm{associative ~algebra} \ar[rr]^{\text{Example ~}\ref{exam-ass-assy}} \ar[d]_{\text{skew-symmetrization}} & & \text{associative-Yamaguti ~ algebra}  \ar[d]^{\text{Theorem ~}\ref{assy-liey}} \\
    \mathrm{Lie ~ algebra }  \ar[rr]_{\text{Example ~}\ref{exam-lie-liey}} & & \text{Lie-Yamaguti ~algebra}. 
    }
\end{align}
\item [(ii)] It follows that the construction of an associative-Yamaguti algebra from a diassociative algebra given in Theorem \ref{diass-assy} is natural in the sense that it is also compatible with the construction of a Lie-Yamaguti algebra from a (left) Leibniz algebra. More precisely, the following diagram is also commutative:
\begin{align}\label{diass-diagram}
   \xymatrix{
     \mathrm{diassociative ~algebra} \ar[rr]^{\text{Theorem ~}\ref{diass-assy}} \ar[d]_{\lfloor a, b \rfloor = a ~ \! \vdash ~ \! b - b ~ \! \dashv ~ \! a} & & \text{associative-Yamaguti ~ algebra}  \ar[d]^{\text{Theorem ~}\ref{assy-liey}} \\
    \mathrm{Leibniz ~ algebra }  \ar[rr]_{\text{Example ~}\ref{exam-leib-liey}} & & \text{Lie-Yamaguti ~algebra}. 
    }
\end{align}
    \end{itemize}
\end{remark}

%This functor $\mathcal{S}$ also fits into the following diagram of various categories:
%\begin{align*}
%    \xymatrix{
%    {\bf Ass ~} \ar[d] \ar@{^{(}->}[r] & {\bf AssY} \ar[d]^{\mathcal{S}} \\
%    {\bf Lie~ } \ar@{^{(}->}[r] & {\bf LieY}.
%    }
%\end{align*}

In a particular case of the above theorem, we obtain the following well-known result \cite{meyberg}.

\begin{prop}
    Let $(A, \{ ~, ~ , ~ \})$ be an associative triple system. Then the pair $(A, \llbracket ~, ~ , ~ \rrbracket)$ is a Lie triple system, where
    \begin{align*}
        \llbracket a, b, c \rrbracket := \{ a, b, c \} - \{ b, a , c \} - \{ c, a, b \} + \{ c, b, a \}, \text{ for  } a, b, c \in A.
    \end{align*}
\end{prop}

%The above construction of an associative-Yamaguti algebra from a diassociative algebra fits into the following commutative diagram:

\section{Representations and the (2,3)-th cohomology of associative-Yamaguti algebras}\label{sec4}

In this section, we first define representations of an associative-Yamaguti algebra and provide some examples. Then we introduce the $(2,3)$-th cohomology group of an associative-Yamaguti algebra with coefficients in a representation. Some applications of this cohomology group are discussed in the next section.

Let $(A, ~ \! \cdot ~ \!, \{ ~, ~ , ~ \}, \{ \! \! \{ ~, ~, ~ \} \! \! \})$ be an associative-Yamaguti algebra and $M$ be any vector space (not necessarily having any additional structure). Let $\mathcal{A}^{k, 1}$ be the direct sum of all possible $(k+1)$ tensor powers of $A$ and $M$ in which $A$ appears $k$ times (and thus $M$ appears exactly once). For example,
\begin{align*}
    \mathcal{A}^{1,1} = (A \otimes M) \oplus (M \otimes A) ~~~~ \text{ and } ~~~~ \mathcal{A}^{2,1} = (A \otimes A \otimes M) \oplus (A \otimes M \otimes A) \oplus (M \otimes A \otimes A) ~~ \mathrm{etc.}
\end{align*}
Thus, a linear map $\cdot : \mathcal{A}^{1,1} \rightarrow M$ can be decomposed into two linear maps (both being denoted by the same notation) $\cdot : A \otimes M \rightarrow M$ and $\cdot : M \otimes A \rightarrow M$. Similarly, a linear map $\{ ~, ~ , ~ \} :\mathcal{A}^{2,1} \rightarrow M$ can be decomposed into three linear maps (all being denoted by the same notation)
\begin{align*}
    \{ ~, ~ , ~ \} : A \otimes A \otimes M \rightarrow M, \quad \{ ~, ~ , ~ \} : A \otimes M \otimes A \rightarrow M ~~~~ \text{ and } ~~~~ \{ ~, ~, ~ \} : M \otimes A \otimes A \rightarrow M.
\end{align*}

\begin{defn}
    Let $(A, ~ \! \cdot ~ \!, \{ ~, ~ , ~ \}, \{ \! \! \{ ~, ~, ~ \} \! \! \})$ be an associative-Yamaguti algebra. A {\bf representation} of this associative-Yamaguti algebra is a vector space $M$ equipped with linear maps (also denoted by the same notations as the structure operations of $A$)
    \begin{align}\label{aya-rep}
        \cdot : \mathcal{A}^{1,1} \rightarrow M, \quad \{ ~, ~, ~ \} : \mathcal{A}^{2,1} \rightarrow M ~~~~ \text{ and } ~~~~ \{ \! \! \{ ~, ~, ~ \} \! \! \} : \mathcal{A}^{2,1} \rightarrow M
    \end{align}
    that satisfy the identities (\text{\ref{ay1}})-(\text{\ref{ay11}}) when exactly one of the variables be in $M$ and others in $A$.
\end{defn}

It follows that the operations given in (\ref{aya-rep}) satisfy a list of 58 identities (3 identities obtained from (\text{\ref{ay1}}), 4 identities obtained from each (\text{\ref{ay2}})-(\text{\ref{ay6}}), 10 identities obtained from each (\text{\ref{ay7}}), (\text{\ref{ay9}}), and 5 identities obtained from each (\text{\ref{ay8}}), (\text{\ref{ay10}}), (\text{\ref{ay11}})).

The following result can characterize representations of an associative-Yamaguti algebra.

\begin{prop}\label{prop-assy-semid}
    Let $(A, ~ \! \cdot ~ \!, \{ ~, ~ , ~ \}, \{ \! \! \{ ~, ~, ~ \} \! \! \})$ be an associative-Yamaguti algebra and $M$ be a vector space equipped with linear maps as of (\ref{aya-rep}). For any $(a, u), (b, v), (c, w) \in A \oplus M$, we define
    \begin{align*}
        (a, u) \cdot_\ltimes (b, v) :=~& ( a \cdot b ~ \!, ~ \! a \cdot v + u \cdot b), \\
        \{ (a, u) , (b, v) , (c, w) \}_\ltimes :=~& (\{ a, b, c \} ~ \!, ~ \! \{ a, b, w \} + \{ a, v, c \} + \{u, b, c \} ),\\
         \! \! \! \{ (a, u) , (b, v) , (c, w) \} \! \! \}_\ltimes :=~& (\{ \! \! \{ a, b, c \} \! \! \} ~ \!, ~ \! \{ \! \! \{ a, b, w \} \! \! \} + \{ \! \! \{ a, v, c\} \! \! \} + \{ \! \! \{u, b, c \} \! \! \} ).
    \end{align*}
    Then $(A \oplus M, ~ \! \cdot_\ltimes ~ \!, \{ ~, ~ , ~ \}_\ltimes, \{ \! \! \{ ~, ~, ~ \} \! \! \}_\ltimes)$ is an associative-Yamaguti algebra if and only if the space $M$ with the maps (\ref{aya-rep}) is a representation of $(A, ~ \! \cdot ~ \!, \{ ~, ~ ,~ \}, \{ \! \! \{ ~, ~, ~ \} \! \! \})$.
\end{prop}

\begin{exam}
    Let $(A, ~ \! \cdot ~ \!)$ be an associative algebra and $(M, ~ \! \cdot ~ \! )$ be a bimodule over it. Define linear maps $\{ ~, ~ , ~ \} , \{ \! \! \{ ~, ~, ~ \} \! \! \} : \mathcal{A}^{2,1} \rightarrow M$ by
    \begin{align*}
        \{ a, b, u \} =~& \{ \! \! \{ a, b, u \} \! \! \} = (a \cdot b) \cdot u = a \cdot (b \cdot u), \qquad \{ a, u, b \} = \{ \! \! \{ a, u, b \} \! \! \} = (a \cdot u) \cdot b = a \cdot (u \cdot b), \\
       & \text{ and } \{ u, a, b \} = \{ \! \! \{ u, a, b \} \! \! \} = (u \cdot a) \cdot b = u \cdot (a \cdot b), \text{ for } a, b \in A, u \in M.
    \end{align*}
    Then the quadruple $(M, ~ \! \cdot ~ \!, \{ ~, ~, ~ \}, \{ \! \! \{ ~, ~, ~ \} \! \! \})$ is a representation of the associative-Yamaguti algebra $(A, ~ \! \cdot ~ \!, \{ ~, ~, ~ \}, \{ \! \! \{ ~, ~, ~ \} \! \! \})$ given in Example \ref{exam-ass-assy}.
\end{exam}

\begin{exam}
    Let $(A = A_0 \oplus A_1, ~ \! \cdot ~ \!)$ be a reductive associative algebra. A {\em reductive bimodule} over it, is a bimodule $(M = M_0 \oplus M_1, ~ \! \cdot ~ \!)$ over the associative algebra $(A = A_0 \oplus A_1, ~ \! \cdot ~ \!)$ for which
    \begin{align*}
        A_0 \cdot M_0 \subset M_0, \quad M_0 \cdot A_0 \subset M_0, \quad A_0 \cdot M_1 \subset M_1, \quad A_1 \cdot M_0 \subset M_1, \quad M_1 \cdot A_0 \subset M_1, \quad M_0 \cdot A_1 \subset M_1.
    \end{align*}
    For any $a, b \in A_1$ and $u \in M_1$, we define
    \begin{align*}
        &a \bullet u = \mathrm{pr}_{M_1} (a \cdot u), \quad u \bullet a = \mathrm{pr}_{M_1} (u \cdot a), \\
        &\{ a, b, u \} = (\mathrm{pr}_{A_0} (a \cdot b)) \cdot u, \qquad \{ a, u, b \} = (\mathrm{pr}_{M_0} (a \cdot u)) \cdot b, \qquad \{ u, a, b \} = (\mathrm{pr}_{M_0} (u \cdot a) )\cdot b,\\
       & \{ \! \! \{ a, b, u   \} \! \! \} =  a \cdot ( \mathrm{pr}_{M_0} (b \cdot u)), \qquad  \{ \! \! \{ a, u, b   \} \! \! \} =  a \cdot (\mathrm{pr}_{M_0} (u \cdot b)), \qquad  \{ \! \! \{ u, a, b   \} \! \! \} = u \cdot (\mathrm{pr}_{A_0} (a \cdot b)).
    \end{align*}
    Then the quadruple $(M_1, ~ \! \bullet ~ \! , \{ ~, ~, ~ \}, \{ \! \! \{ ~, ~, ~ \} \! \! \})$ is a representation of the associative-Yamaguti algebra $(A_1, ~ \! \bullet ~ \!, \{ ~, ~, ~ \}, \{ \! \! \{ ~, ~, ~ \} \! \! \})$ given in Example \ref{reductive}.
\end{exam}

\begin{exam}
    Let $(A, ~ \! \cdot ~ \!, \{ ~, ~, ~ \}, \{ \! \! \{ ~, ~, ~ \} \! \! \})$ be any associative-Yamaguti algebra. Then the quadruple $(A, ~ \! \cdot_{ad} ~ \! , \{ ~, ~, ~ \}_{ad}, \{ \! \! \{ ~, ~ ,~ \} \! \! \}_{ad})$ is obviously a representation, where
    \begin{align*}
        a \cdot_{ad} b = a \cdot b, \quad \{ a, b, c \}_{ad} = \{ a, b, c \}, \quad \{ \! \! \{ a, b, c \} \! \! \}_{ad} = \{ \! \! \{ a, b, c \} \! \! \}, \text{ for } a, b, c \in A.
    \end{align*}
    This is called the {\em adjoint representation} or the {\em regular representation}.
\end{exam}

\begin{exam}
    Let $(A, ~ \! \cdot ~ \!, \{ ~, ~, ~ \}, \{ \! \! \{ ~, ~, ~ \} \! \! \})$ and $(A', ~ \! \cdot' ~ \!, \{ ~, ~, ~ \}', \{ \! \! \{ ~, ~, ~ \} \! \! \}')$ be two associative-Yamaguti algebras, and $\varphi : A \rightarrow A'$ be a homomorphism between them. Then the tuple $(A', ~ \! \cdot_\varphi ~ \! , \{ ~, ~ , ~ \}_\varphi, \{ \! \! \{ ~, ~, ~ \} \! \! \}_\varphi)$ is a representation of the associative-Yamaguti algebra $(A, ~ \! \cdot ~ \!, \{ ~, ~, ~ \}, \{ \! \! \{ ~, ~, ~ \} \! \! \})$, where for $a, b \in A$ and $x \in A'$,
    \begin{align*}
        &a \cdot_\varphi x = \varphi (a) \cdot' x, \quad x \cdot_\varphi a = x \cdot' \varphi (a), \\
        &\{ a, b, x \}_\varphi = \{ \varphi (a), \varphi (b), x \}', \quad \{ a, x, b \}_\varphi = \{ \varphi (a), x, \varphi (b) \}', \quad \{ x, a, b \}_\varphi = \{ x, \varphi (a), \varphi (b) \}',\\
        &\{ \! \! \{ a, b, x \} \! \! \}_\varphi = \{ \! \! \{ \varphi (a), \varphi (b), x  \} \! \! \}', \quad \{ \! \! \{ a, x, b  \} \! \! \}_\varphi = \{ \! \! \{ \varphi (a), x, \varphi (b)  \} \! \! \}', \quad \{ \! \! \{ x, a, b  \} \! \! \}_\varphi = \{ \! \! \{ x, \varphi (a), \varphi (b)  \} \! \! \}'.
    \end{align*}
\end{exam}

\begin{exam}
    Let $(A, \{ ~, ~, ~ \})$ be an associative triple system. Recall that \cite{carlsson} a {\em representation} of $(A, \{ ~, ~, ~ \})$ is a vector space $M$ endowed with a linear map (denoted by the same notation) $\{ ~, ~ , ~ \} : \mathcal{A}^{2,1} \rightarrow M$ satisfying total $10$ identities, which are obtained from the two identities given in (\ref{asst-iden}) by letting exactly one of the variables $a, b, c, d, e$ be in $M$ and others in $A$. A representation of the associative triple system $(A, \{ ~, ~, ~ \})$ can be seen as a representation of the corresponding associative-Yamaguti algebra $(A, ~ \! \cdot ~ \! , \{ ~, ~, ~ \}, \{ ~, ~ , ~ \})$ given in Example \ref{exam-asst}.
\end{exam}

\begin{exam}
    Let $(D, \dashv, \vdash)$ be a diassociative algebra. Recall that \cite{loday,frab} a {\em representation} of $(D, \dashv, \vdash)$ is a vector space $M$ equipped with four linear maps
    \begin{align*}
        \dashv ~: D \otimes M \rightarrow M, \quad \dashv~ : M \otimes D \rightarrow M, \quad \vdash~ : D \otimes M \rightarrow M, \quad \vdash~ : M \otimes D \rightarrow M
    \end{align*}
    so that the first two operations make $M$ into a bimodule over the associative algebra $(D, \dashv)$ and the last two operations make $M$ into a bimodule over the associative algebra $(D, \vdash)$ satisfying additionally the identities in (\ref{dialg-iden}) when exactly one of the variable be in $M$ and others in $D$. In this case, we define maps
    \begin{align*}
        &a \cdot u = a \dashv u + a \vdash u, \qquad u \cdot a = u \dashv a + u \vdash a, \\
        &\{ a, b, u \} = - (a \vdash b) \vdash u, \quad \{ a, u, b \} = - (a \vdash u) \vdash b, \quad \{ u, a, b \} = - (u \vdash a) \vdash b, \\
        & \{ \! \! \{  a, b, u  \} \! \! \} = - ( a \dashv b ) \dashv u,  \quad  \{ \! \! \{ a, u, b   \} \! \! \} = - (a  \dashv  u) \dashv b, \quad  \{ \! \! \{ u, a, b   \} \! \! \} = - ( u \dashv  a) \dashv  b,  
    \end{align*}
    for $a, b \in D$ and $u \in M$. Then the quadruple $(M, ~ \! \cdot ~ \! , \{ ~, ~, ~ \}, \{ \! \! \{ ~, ~, ~ \} \! \! \})$ is a representation of the associative-Yamaguti algebra $(D, ~  \! \cdot ~ \!, \{ ~, ~, ~ \}, \{ \! \! \{ ~, ~, ~ \} \! \! \})$ obtained in Theorem \ref{diass-assy}.
\end{exam}

\begin{prop}
    Let $(A, ~ \! \cdot ~ \!, \{ ~, ~,  ~ \}, \{ \! \! \{ ~, ~, ~ \} \! \! \})$ be an associative-Yamaguti algebra and $(M, ~ \! \cdot ~ \! , \{ ~,~ , ~ \}, \{ \! \! \{ ~, ~, ~ \} \! \! \})$ be a representation of it. We define linear maps $\rho : A \rightarrow \mathrm{End}(M)$ and $\nu : A \otimes A \rightarrow \mathrm{End}(M)$ by
    \begin{align}\label{rho-mu}
        \rho (a) u := a \cdot u - u \cdot a  \quad \text{ and } \quad
        \nu(a, b) u := \{ u, a, b \} - \{ a, u, b \} - \{ \! \! \{ b, u, a \} \! \! \} + \{ \! \! \{ b, a, u \} \! \! \},
    \end{align}
    for $a, b \in A$ and $u \in M$. Then the triple $(M, \rho, \nu)$ is a representation of the Lie-Yamaguti algebra $(A, [~,~], \llbracket ~, ~, ~ \rrbracket)$ given in Theorem \ref{assy-liey}.
\end{prop}

\begin{proof}
    Since $(M, ~ \! \cdot ~ \! , \{ ~,~ , ~ \}, \{ \! \! \{ ~, ~, ~ \} \! \! \})$ is a representation of $(A, ~ \! \cdot ~ \!, \{ ~, ~ ,~ \}, \{ \! \! \{ ~, ~, ~ \} \! \! \})$, one may construct the semidirect product associative-Yamaguti algebra $(A \oplus M, ~ \! \cdot_\ltimes ~ \! , \{ ~, ~, ~ \}_\ltimes, \{ \! \! \{ ~, ~, ~ \} \! \! \}_\ltimes)$ given in Proposition \ref{prop-assy-semid}. Hence by following Theorem \ref{assy-liey}, one have a Lie-Yamaguti algebra $(A \oplus M, [~,~]_\ltimes , \llbracket ~, ~ , ~ \rrbracket_\ltimes)$, where
    \begin{align*}
        [ (a, u), (b, v)]_\ltimes :=~& (a, u) \cdot_\ltimes (b, v) ~ - ~ (b, v) \cdot_\ltimes (a, u), \\
        \llbracket (a, u), (b, v) , (c, w) \rrbracket_\ltimes :=~& \{ (a, u), (b, v) , (c, w) \}_\ltimes - \{  (b, v), (a, u), (c, w) \}_\ltimes \\
        & - \{ \! \! \{ (c, w), (a, u), (b, v)    \} \! \! \}_\ltimes ~+~ \{ \! \! \{  (c, w), (b, v), (a, u)   \} \! \! \}_\ltimes,
    \end{align*}
    for $(a, u), (b, v), (c, w) \in A \oplus M$. By expanding the right-hand sides of the above expressions, it is not hard to see that
    \begin{align*}
        [ (a, u), (b, v)]_\ltimes =~& ( [a, b], \rho (a) v - \rho (b) u),\\
        \llbracket (a, u), (b, v) , (c, w) \rrbracket_\ltimes =~& ( \llbracket a, b, c \rrbracket, D_{\rho, \nu} (a, b) w + \nu (b, c) u - \nu(a, c) v),
    \end{align*}
    where the maps $\rho$ and $\nu$ (and hence $D_{\rho, \nu}$) are given in (\ref{rho-mu}). Hence by Proposition \ref{prop-zhang}, one obtains that $(M, \rho, \nu)$ is a representation of the Lie-Yamaguti algebra $(A, [~,~], \llbracket ~, ~, ~ \rrbracket)$.
\end{proof}

%\section{The $(2,3)$-cohomology group of an associative-Yamaguti algebra and applications} 

%In this section, we define the $(2,3)$-cohomology group of an associative-Yamaguti algebra with coefficients in a representation. Subsequently, we study deformations of an associative-Yamaguti algebra in terms of this cohomology group with coefficients in the adjoint representation. Finally, we consider abelian extensions of an associative-Yamaguti algebra by a given representation and show that the set of all isomorphism classes of such abelian extensions is classified by the $(2,3)$-cohomology group.

In the following, we define the $(2,3)$-th cohomology group of an associative-Yamaguti algebra with coefficients in a representation. However, the full cochain complex of an associative-Yamaguti algebra has yet to be discovered.

\begin{defn}
    Let $(A, ~ \! \cdot ~ \! , \{ ~, ~ , ~ \}, \{ \! \! \{ ~, ~, ~ \} \! \! \})$ be an associative-Yamaguti algebra and $(M, ~ \! \cdot ~ \!, \{ ~, ~ ,~ \}, \{ \! \! \{ ~, ~, ~ \} \! \! \})$ be a representation. A {\bf (2,3)-cocycle} of $A$ with coefficients in $M$ is a triple $(\mu, F, G)$ consisting of linear maps $\mu : A^{\otimes 2} \rightarrow M$ and $F, G : A^{\otimes 3} \rightarrow M$ that satisfy the following list of identities:
    \begin{equation}
        \mu (a, b ) \cdot c + \mu (a \cdot b , c) - a \cdot \mu (b, c) - \mu (a, b \cdot c) + F (a, b, c) - G (a, b, c) = 0,       
    \end{equation}
    \begin{equation}
  \{ \mu (a, b) , c, d \} + F (a \cdot b, c, d) = \{ a, \mu (b, c), d \} + F (a, b \cdot c, d),      
    \end{equation}
    \begin{equation}
      \{ a, b, \mu (c, d) \} + F (a, b, c \cdot d) = \mu (\{ a, b, c \} , d) + F (a, b, c) \cdot d,  
    \end{equation}
    \begin{equation}
      \{ \! \! \{ \mu (a, b), c, d \} \! \! \} + G (a \cdot b, c, d) = \mu (a, \{ \! \! \{ b, c, d \} \! \! \}) + a \cdot G (b, c, d),  
    \end{equation}
    \begin{equation}
        \{ \! \! \{ a, \mu (b, c) , d \} \! \! \} + G (a, b \cdot c, d) = \{ \! \! \{ a, b, \mu (c, d) \} \! \! \} + G (a, b, c \cdot d),
    \end{equation}
    \begin{equation}
       \mu (a, \{ b, c, d \}) + a \cdot F (b, c, d) = \mu ( \{ \! \! \{ a, b, c \} \! \! \}, d) + G (a, b, c) \cdot d, 
    \end{equation}
    \begin{equation}
          \{ F (a, b, c ) , d, e \} + F (\{ a, b, c \}, d, e) = \{ a, G (b, c, d), e \} + F (a, \{ \! \! \{ b, c, d \} \! \! \} , e) = \{a, b, F( c, d, e ) \} + F (a, b, \{ c, d, e \}),
    \end{equation}
    \begin{equation}
         \{ a, F (b, c, d ), e \} + F  (a, \{b, c, d \}, e) = \{ G (a, b, c ), d,e \} + F (\{ \! \! \{ a, b, c \} \! \! \}, d,e ),
    \end{equation}
    \begin{equation}
       \{ \! \! \{  G (a, b, c), d, e       \} \! \! \} + G ( \{ \! \! \{  a, b, c       \} \! \! \}, d, e) =  \{ \! \! \{  a, F (b, c, d), e       \} \! \! \} + G (a, \{ b, c, d \} , e ) =  \{ \! \! \{  a, b, G( c, d, e)       \} \! \! \} + G (a, b,  \{ \! \! \{   c, ,d ,e      \} \! \! \}), 
    \end{equation}
    \begin{equation}
        \{ \! \! \{    a, G (b, c, d) , e    \} \! \! \} + G (a,  \{ \! \! \{  b, c, d       \} \! \! \}, e) =  \{ \! \! \{ a, b, F (c, d, e)        \} \! \! \} + G (a,b, \{c, d, e \}),
    \end{equation}
    \begin{equation}
        \{ a, b, G (c, d, e) \} + F (a, b,  \{ \! \! \{  c, d,e        \} \! \! \}) =  \{ \! \! \{   F (a, b, c), d, e      \} \! \! \} + G (\{ a, b, c \}, d, e),
    \end{equation}
    for all $a, b, c, d, e \in A$. The space of all $(2,3)$-cocycles is an abelian group under componentwise addition. We denote this space by $\mathcal{Z}^ {(2,3)} (A, M)$.
\end{defn}

The next result provides a characterization of $(2,3)$-cocycles, and the proof is straightforward.

\begin{prop}\label{twisted semid}
    Let $(A, ~ \! \cdot ~ \! , \{ ~, ~ , ~ \}, \{ \! \! \{ ~, ~, ~ \} \! \! \})$ be an associative-Yamaguti algebra and $(M, ~ \! \cdot ~ \!, \{ ~, ~ , ~ \}, \{ \! \! \{ ~, ~, ~ \} \! \! \})$ be a representation of it. Suppose there are linear maps $\mu : A^{\otimes 2} \rightarrow M$ and $F, G : A^{\otimes 3} \rightarrow M$. For any $(a, u) , (b, v) , (c, w) \in A \oplus M$, we define
    \begin{align}
        (a, u) \cdot_{\ltimes_\mu} (b, v) :=~& \big( a \cdot b ~ \!, ~ \! a \cdot v + u \cdot b + \mu (a, b) \big), \label{twisted-mu}\\
        \{ (a, u), (b, v), (c, w ) \}_{\ltimes_F} :=~& \big( \{ a, b, c \} ~ \!, ~ \! \{ a, b, w \} +  \{ a, v, c\}  +  \{u, b, c \} + F (a, b, c) \big),\\ 
        \{ \! \! \{ (a, u), (b, v), (c, w )  \} \! \! \}_{\ltimes_G} :=~&  \big(\{ \! \! \{ a, b, c \} \! \! \} ~ \!, ~ \! \{ \! \! \{ a, b, w \} \! \! \} + \{ \! \! \{ a, v, c\} \! \! \} + \{ \! \! \{u, b, c \} \! \! \} + G (a, b, c) \big).
    \end{align}
    Then $(\mu, F, G)$ is a $(2,3)$-cocycle if and only if the triple $(A \oplus M, ~ \! \cdot_{\ltimes_\mu}  ~ \!, \{ ~, ~, ~ \}_{\ltimes_F} ,   \{ \! \! \{ ~, ~, ~  \} \! \! \}_{\ltimes_G}  )$ is an associative-Yamaguti algebra.
\end{prop}

Let $f : A \rightarrow M$ be a linear map. Then we define linear maps $\mu_f : A^{\otimes 2} \rightarrow M$ and $F_f, G_f : A^{\otimes 3} \rightarrow M$ by
\begin{align}\label{cob-cocycle}
\begin{cases}
    \mu_f (a, b) := f(a) \cdot b + a \cdot f(b) - f (a \cdot b),\\
    F_f (a, b, c ) := \{ f(a), b, c \} + \{ a, f(b) , c \} + \{ a, b, f(c) \} - f (\{ a, b, c \}),\\
    G_f (a, b, c) := \{ \! \! \{  f(a) , b, c   \} \! \! \} +  \{ \! \! \{ a, f(b) , c    \} \! \! \} +  \{ \! \! \{  a, b, f(c)   \} \! \! \} - f ( \{ \! \! \{ a, b, c    \} \! \! \}),
\end{cases}
\end{align}
for $a, b, c \in A$. A linear map $f : A \rightarrow M$ is said to be a {\bf derivation} on $A$ with values in $M$ if the triple $(\mu_f, F_f , G_f)$ is trivial. In general, we have the following.
\begin{prop}
    For any linear map $f : A \rightarrow M$, the triple $(\mu_f, F_f, G_f)$ is a $(2,3)$-cocycle.
\end{prop}

\begin{proof}
Let $a, b, c, d, e \in A$ be arbitrary. Then by using the definitions of $\mu_f$, $F_f$ and $G_f$, we simply get
\begin{align*}
    &\mu_f (a, b) \cdot c + \mu_f (a \cdot b, c) - a \cdot \mu_f (b, c) - \mu_f (a, b \cdot c) + F_f (a, b, c) - G_f (a, b, c) \\
    &= \big(  (f(a) \cdot b) \cdot c - f(a) \cdot (b \cdot c) + \{ f(a), b, c \} - \{ \! \! \{ f(a) , b , c \} \! \! \}  \big) + \big(    (a \cdot f(b)) \cdot c - a \cdot (f(b) \cdot c)  \\
    & \qquad + \{ a, f(b), c \} - \{ \! \! \{ a, f(b) , c \} \! \! \}  \big) + \big(   (a \cdot b) \cdot f(c) - a \cdot (b \cdot f(c)) + \{ a, b, f(c) \} - \{ \! \! \{ a, b, f(c) \} \! \! \}   \big) \\
    & \qquad - f \big(   (a \cdot b) \cdot c - a \cdot (b \cdot c) + \{ a, b, c \} - \{ \! \! \{ a, b, c \} \! \! \}   \big) \stackrel{(\text{\ref{ay1}})}{=} 0,
\end{align*}
\begin{align*}
    &\{ \mu_f(a, b) , c, d \} + F_f (a \cdot b , c, d) - \{ a, \mu_f (b, c) , d \} - F_f (a, b \cdot c, d) \\
   & = \big(   \{  f( a) \cdot b, c, d \} - \{ f(a), b \cdot c, d \}  \big) + \big(   \{   a \cdot f( b), c, d \} - \{ a, f(b) \cdot c, d \}  \big) + \big(   \{   a \cdot b, f(c), d \} - \{ a, b \cdot f(c), d \}  \big) \\
    & \qquad \qquad + \big(   \{   a \cdot b, c, f(d) \} - \{ a, b \cdot c, f( d) \}  \big) - f \big(   \{   a \cdot b, c, d \} - \{ a, b \cdot c, d \}  \big) \stackrel{(\text{\ref{ay2}})}{=} 0,
\end{align*}
\begin{align*}
    &\{ a, b, \mu_f (c, d) \} + F_f (a, b, c\cdot d) - \mu_f (\{ a, b, c \}, d) - F_f (a, b, c) \cdot d \\
    &= \big(  \{ f(a), b, c \cdot d \} - \{ f(a), b, c \} \cdot d   \big) + \big(  \{ a, f(b), c \cdot d \} - \{ a, f(b), c \} \cdot d   \big) + \big(  \{ a, b, f(c) \cdot d \} - \{ a, b, f(c) \} \cdot d   \big) \\
    & \qquad \qquad + \big(  \{ a, b, c \cdot f(d) \} - \{ a, b, c \} \cdot f(d)   \big) - f \big(  \{ a, b, c \cdot d \} - \{ a, b, c \} \cdot d   \big) \stackrel{(\text{\ref{ay3}})}{=} 0,
\end{align*}
\begin{align*}
    &\{ \! \! \{  \mu_f (a, b) , c, d \} \! \! \} + G_f (a \cdot b, c, d ) - \mu_f (a, \{ \! \! \{ b, c, d \} \! \! \}) - a \cdot G_f (b, c, d ) \\
    &= \big(    \{ \! \! \{ f(a) \cdot b, c, d  \} \! \! \} -  f(a) \cdot  \{ \! \! \{ b, c, d \} \! \! \} \big) + \big(    \{ \! \! \{ a \cdot f(b), c, d  \} \! \! \} -  a \cdot  \{ \! \! \{ f(b), c, d \} \! \! \} \big)  \\
    & \qquad + \big(    \{ \! \! \{ a \cdot b, f(c), d  \} \! \! \} -  a \cdot  \{ \! \! \{ b, f( c), d \} \! \! \} \big) + \big(    \{ \! \! \{ a \cdot b, c, f(d)  \} \! \! \} -  a \cdot  \{ \! \! \{ b, c, f(d) \} \! \! \} \big) \\
    & \qquad \qquad - f \big(    \{ \! \! \{ a \cdot b, c, d  \} \! \! \} -  a \cdot  \{ \! \! \{ b, c, d \} \! \! \} \big) \stackrel{(\text{\ref{ay4}})}{=} 0,
\end{align*}
\begin{align*}
    &\{ \! \! \{ a, \mu_f (b, c), d \} \! \! \} + G_f (a, b \cdot c, d) - \{ \! \! \{a, b, \mu_f (c, d) \} \! \! \} - G_f (a, b, c \cdot d) \\
    &= \big(  \{ \! \! \{ f(a), b \cdot c , d \}\! \! \} - \{ \! \! \{ f(a), b, c \cdot d \} \! \! \} \big) + \big(  \{ \! \! \{ a, f(b) \cdot c , d \}\! \! \} - \{ \! \! \{ a, f(b), c \cdot d \} \! \! \} \big)\\
    & \qquad + \big(  \{ \! \! \{ a, b \cdot f(c) , d \}\! \! \} - \{ \! \! \{ a, b, f(c) \cdot d \} \! \! \} \big) + \big(  \{ \! \! \{ a, b \cdot c , f(d) \}\! \! \} - \{ \! \! \{ a, b, c \cdot f(d) \} \! \! \} \big) \\
    & \qquad \qquad - f \big(  \{ \! \! \{ a, b \cdot c , d \}\! \! \} - \{ \! \! \{ a, b, c \cdot d \} \! \! \} \big) \stackrel{(\text{\ref{ay5}})}{=} 0,
\end{align*}
\begin{align*}
    &\mu_f (a, \{ b, c, d \}) + a \cdot F_f (b, c, d) - \mu_f (\{ \! \! \{ a, b, c \} \! \! \}, d) - G_f (a, b, c) \cdot d \\
    & = \big(   f(a) \cdot \{ b, c, d \} - \{ \! \! \{ f(a), b, c \} \! \! \} \cdot d  \big) + \big(   a \cdot \{ f(b), c, d \} - \{ \! \! \{ a, f(b), c \} \! \! \} \cdot d  \big) + \big(   a \cdot \{ b, f(c), d \} - \{ \! \! \{ a, b, f(c) \} \! \! \} \cdot d  \big) \\
    & \qquad + \big(   a \cdot \{ b, c, f(d) \} - \{ \! \! \{ a, b, c \} \! \! \} \cdot f(d)  \big) - f \big(   a \cdot \{ b, c, d \} - \{ \! \! \{ a, b, c \} \! \! \} \cdot d  \big) \stackrel{(\text{\ref{ay6}})}{=} 0.
\end{align*}
Thus, we have verified the first 6 identities of a $(2,3)$-cocycle. The remaining identities can be proved similarly.
\end{proof}

%Let  $(A, ~ \! \cdot ~ \!, \{ ~, ~, ~ \}, \{ \! \! \{ ~, ~, ~ \} \! \! \})$ be an associative-Yamaguti algebra and $(M, ~ \! \cdot ~ \!, \{ ~, ~ ,~ \}, \{ \! \! \{ ~, ~, ~ \} \! \! \})$ be a representation. Then a linear map $\mathcal{D} : A \rightarrow M$ is said to be a {\bf derivation} on $A$ with values in $M$ if
%\begin{align*}
%    \mathcal{D}  (a \cdot b) =~& \mathcal{D} (a) \cdot b + a \cdot \mathcal{D} (b),\\
%    \mathcal{D}  \{ a, b, c \} =~& \{ \mathcal{D} (a), b , c \} + \{ a, \mathcal{D} (b), c \} + \{ a, b, \mathcal{D} (c) \},\\
%    \mathcal{D}  \{ \! \! \{ a , b , c \} \! \! \} =~& \{ \! \! \{ \mathcal{D} (a) , b , c \} \! \! \} + \{ \! \! \{ a , \mathcal{D} (b) , c \} \! \! \} + \{ \! \! \{  a , b ,  \mathcal{D} (c) \} \! \! \}, \text{ for all } a, b, c \in A.
%\end{align*}
%Thus, a linear map $\mathcal{D} : A \rightarrow M$ is a derivation if and only if the $(2,3)$-cocycle $(\mu_\mathcal{D}, F_\mathcal{D}, G_\mathcal{D})$ is trivial.

A $(2,3)$-cocycle $(\mu, F, G)$ is said to be a {\bf $(2,3)$-coboundary} if there exists a linear map $f: A \rightarrow M$ such that $(\mu, F, G) = (\mu_f, F_f, G_f)$. The space of all $(2,3)$-coboundaries is denoted by the notation $\mathcal{B}^{(2,3)} (A, M)$. Then the quotient space
\begin{align*}
    \mathcal{H}^{(2,3)}(A, M) := \frac{ \mathcal{Z}^{(2,3)} (A, M)  }{ \mathcal{B}^{(2,3)} (A, M) }
\end{align*}
is said to be the {\bf $(2,3)$-th cohomology group} of the associative-Yamaguti algebra $A$ with coefficients in the representation $M.$

\section{Deformations and abelian extensions of associative-Yamaguti algebras}\label{sec5}

In this section, we study formal deformations and abelian extensions of an associative-Yamaguti algebra in terms of its $(2,3)$-th cohomology group.

\subsection{Formal deformations} Let $(A, ~ \! \cdot ~ \!, \{ ~, ~ , ~ \}, \{ \! \! \{ ~, ~ , ~ \} \! \! \})$ be an associative-Yamaguti algebra. Consider the space $A[[t]]$ of all formal power series in $t$ with coefficients from $A$. Then $A[[t]]$ is obviously a ${\bf k} [[t]]$-module. A {\bf formal $1$-parameter deformation} or simply a {\em formal deformation} of $(A, ~ \! \cdot ~ \!, \{ ~, ~ , ~ \}, \{ \! \! \{ ~, ~ , ~ \} \! \! \})$ consists of ${\bf k} [[t]]$-linear operations 
\begin{align*}
   & \mu_t : A[[t]] \otimes_{ {\bf k} [[t]]} A[[t]] \rightarrow A[[t]], \quad \mu_t (a, b ) = a \cdot b + t \mu_1 (a, b) + t^2 \mu_2 (a, b) + \cdots, \\
   & F_t : A[[t]] \otimes_{ {\bf k} [[t]]} A[[t]] \otimes_{ {\bf k} [[t]]} A[[t]]   \rightarrow A[[t]], \quad F_t (a, b, c ) = \{ a, b, c \} + t F_1 (a, b, c) + t^2 F_2 (a, b, c) + \cdots, \\
   & G_t : A[[t]] \otimes_{ {\bf k} [[t]]} A[[t]] \otimes_{ {\bf k} [[t]]} A[[t]]   \rightarrow A[[t]], \quad G_t (a, b, c ) = \{ \! \! \{ a , b, c \} \! \! \} + t G_1 (a, b, c) + t^2 G_2 (a, b, c) + \cdots ,
\end{align*}
(where $\mu_i : A^{\otimes 2} \rightarrow A$, $F_i : A^{\otimes 3} \rightarrow A$ and $G_i : A^{\otimes 3} \rightarrow A$ are linear maps for all $i$) that makes the tuple $(A[[t]], \mu_t, F_t, G_t)$ into an associative-Yamaguti algebra over the ring ${\bf k} [[t]]$.

\medskip

Let $(A[[t]], \mu_t, F_t, G_t)$ and $(A[[t]], \mu'_t, F'_t, G'_t)$ be two formal deformations of a given associative-Yamaguti algebra $(A, ~ \! \cdot ~ \! , \{ ~, ~, ~ \}, \{ \! \! \{ ~, ~, ~ \} \! \! \})$. They are said to be {\bf equivalent} if there exists a ${\bf k}[[t]]$-linear map 
\begin{align}\label{phi-t}
    \varphi_t : A[[t]] \rightarrow A[[t]] ~\text{ of the form }~ \varphi_t (a) = a + t \varphi_1 (a) + t^2 \varphi_2 (a) + \cdots 
\end{align}
(where $\varphi_i : A \rightarrow A$ are linear maps for all $i$)  that is a homomorphism of associative-Yamaguti algebras from $(A[[t]], \mu_t, F_t, G_t)$ to $(A[[t]], \mu'_t, F'_t, G'_t)$.

\begin{thm}\label{thm-inf-co}
    Let $(A, ~ \! \cdot ~ \! , \{ ~, ~ , ~ \}, \{ \! \! \{ ~, ~, ~ \} \! \! \})$ be an associative-Yamaguti algebra and $(A[[t]], \mu_t, F_t, G_t)$ a formal deformation of it. Then $(\mu_1, F_1, G_1)$ is a $(2,3)$-cocycle of the associative-Yamaguti algebra $(A, ~ \! \cdot ~ \! , \{ ~, ~ , ~ \}, \{ \! \! \{ ~, ~, ~ \} \! \! \})$ with coefficients in the adjoint representation. 
\end{thm}

\begin{proof}
    Since $(A[[t]], \mu_t, F_t, G_t)$ is an associative-Yamaguti algebra over the ring ${\bf k} [[t]]$, the identities (\ref{ay1})-(\ref{ay11}) must be satisfied for this structure. Thus, for any $n \geq 0$ and $a, b, c , d, e \in A$, we must have
    \begin{align*}
       & \qquad \qquad \qquad \sum_{i +j = n}  \big(~ \mu_i ( \mu_j (a, b), c) - \mu_i (a, \mu_j (b, c)) ~\big) + F_n (a, b, c) - G_n (a, b, c) = 0,\\
        & \sum_{i +j = n} F_i (\mu_j (a, b), c, d) =  \sum_{i +j = n} F_i (a, \mu_j (b, c), d), \quad
          \sum_{i +j = n} F_i (a, b, \mu_j (c, d) ) =  \sum_{i +j = n} \mu_i (F_j (a, b, c), d), \\
         &  \sum_{i +j = n} G_i ( \mu_j (a, b), c, d) =  \sum_{i +j = n} \mu_i (a, G_j (b, c, d)), \quad
            \sum_{i +j = n} G_i (a, \mu_j (b, c), d) =  \sum_{i +j = n} G_i (a, b, \mu_j (c, d)),\\
           & \qquad \qquad \qquad \qquad \qquad \qquad \sum_{i +j = n} \mu_i (a, F_j (b, c, d)) =  \sum_{i +j = n} \mu_i (G_j (a, b, c), d), \\
           &  \qquad \qquad \qquad \sum_{i +j = n} F_i (F_j (a, b, c), d, e) =  \sum_{i +j = n} F_i (a, G_j (b, c, d), e) =  \sum_{i +j = n} F_i (a, b, F_j (c, d, e)), \\
           & \qquad \qquad \qquad \qquad \qquad \qquad \sum_{i +j = n} F_i (a, F_j (b, c, d), e) =  \sum_{i +j = n} F_i (G_j (a, b, c), d, e ), \\
           &  \qquad \qquad \qquad  \sum_{i +j = n} G_i (G_j (a, b, c), d, e) = \sum_{i +j = n}  G_i (a, F_j (b, c, d), e) = \sum_{i +j = n} G_i (a, b, G_j (c, d, e)), \\
           &  \sum_{i+j = n} G_i (a, G_j (b, c, d), e) = \sum_{i+j = n} G_i (a, b, F_j (c, d, e)), ~~  \sum_{i+j = n} F_i (a, b, G_j (c, d, e)) = \sum_{i+j= n} G_i (F_j (a, b, c), d, e),
    \end{align*}
    where we assumed that $\mu_0 = \cdot~$, $F_0 = \{ ~, ~, ~ \}$ and $G_0 = \{ \! \! \{ ~, ~, ~ \} \! \! \}$. For $n=1$, we simply get that $(\mu_1, F_1, G_1)$ is a $(2,3)$-cocycle.
\end{proof}

The $(2,3)$-cocycle $(\mu_1, F_1, G_1)$ is called the {\em infinitesimal} of the given formal deformation. More generally, if $(\mu_1, F_1, G_1)$, $(\mu_2, F_2, G_2), \ldots , (\mu_{k-1}, F_{k-1}, G_{k-1})$ are trivials and $(\mu_k, F_k, G_k)$ is the first nontrivial term then it is a $(2,3)$-cocycle, called the {\em $k$-th infinitesimal}.

\begin{prop}
    Let $(A, ~ \! \cdot ~ \! , \{ ~, ~, ~ \}, \{ \! \! \{ ~, ~, ~ \} \! \! \})$ be an associative-Yamaguti algebra. Then the infinitesimals corresponding to equivalent formal deformations of $A$ are differ by a $(2,3)$-coboundary. 
    Therefore, the infinitesimals corresponding to equivalent deformations give rise to the same element in $\mathcal{H}^{(2,3)}(A, A)$.
\end{prop}

\begin{proof}
    Let $(A[[t]], \mu_t, F_t, G_t)$ and $(A[[t]], \mu'_t, F'_t, G'_t)$ be two equivalent formal deformations, and suppose $\varphi_t : A[[t]] \rightarrow A[[t]]$ defined by (\ref{phi-t}) is a homomorphism between them. Then for any $n \geq 0$ and $a, b, c \in A$,
    \begin{align*}
        \sum_{i+j = n} \varphi_i (\mu_j (a, b)) =& \sum_{i+j+k=n} \mu_i' (\varphi_j (a), \varphi_k (b)), \quad \sum_{i+j = n} \varphi_i (F_j (a, b, c)) = \sum_{i+j+k+l = n} F_i' (\varphi_j (a), \varphi_k (b), \varphi_l (c)), \\
        & \sum_{i+j = n} \varphi_i (G_j (a, b, c)) = \sum_{i+j+k+l = n} G_i' (\varphi_j (a), \varphi_k (b), \varphi_l (c)),
    \end{align*}
    where we assumed that $\varphi_0 = \mathrm{Id}_A$. For $n =1$, these three identities can be equivalently written as 
    \begin{align*}
        (\mu_1, F_1, G_1) - (\mu_1' , F_1' , G_1') = (\mu_{\varphi_1}, F_{\varphi_1}, G_{\varphi_1}).
    \end{align*}
    See the notations from (\ref{cob-cocycle}).
    This proves the desired result.
\end{proof}

\subsection{Abelian extensions} Let $(A, ~ \! \cdot ~ \! , \{ ~, ~ , ~ \}, \{ \! \! \{ ~, ~, ~ \} \! \! \})$ be an associative-Yamaguti algebra and $M$ be any vector space (not necessarily having any additional structure). We regard $M$ as an associative-Yamaguti algebra $(M, 0, 0, 0)$ with the trivial structure operations. An {\bf abelian extension} of $(A, ~ \! \cdot ~ \! , \{ ~, ~ , ~ \}, \{ \! \! \{ ~, ~, ~ \} \! \! \})$ by the vector space $M$ is a new associative-Yamaguti algebra $(E, ~\! \cdot_E ~ \!, \{ ~, ~, ~ \}_E, \{ \! \! \{ ~, ~, ~ \} \! \! \}_E)$ with a short exact sequence
\begin{align}\label{abel-ext-diag}
    \xymatrix{
    0 \ar[r] & M \ar[r]^i & E \ar[r]^p & A \ar[r] & 0
    }
\end{align}
of associative-Yamaguti algebras. We denote an abelian extension as above by $(E, ~\! \cdot_E ~ \!, \{ ~, ~, ~ \}_E, \{ \! \! \{ ~, ~, ~ \} \! \! \}_E) $ or simply by $E$ when the maps $i$ and $p$ are clear from the context. A {\em section} of the map $p$ is a linear map $s : A \rightarrow E$ that satisfies $p  s = \mathrm{Id}_A$. For any section $s$, we define linear maps $\cdot : \mathcal{A}^{1, 1} \rightarrow M$, $\{ ~, ~ , ~ \} : \mathcal{A}^{2, 1} \rightarrow M$ and $\{ \! \! \{ ~, ~, ~ \} \! \! \} : \mathcal{A}^{2, 1} \rightarrow M$ by
\begin{align*}
    & \qquad \qquad \qquad a \cdot u := s(a) \cdot_E u, \quad  u \cdot a := u \cdot_E s(a),\\
     &\{ a, b, u \} := \{ s(a), s(b), u \}_E, ~~~ \{ a, u, b \} :=  \{ s(a), u , s (b) \}_E, ~~ \{ u, a, b \} := \{ u, s(a), s(b) \}_E, \\
    &  \{ \! \! \{ a, b, u \} \! \! \} := \{ \! \! \{ s(a), s(b), u \}  \! \! \}_E, ~~ \{ \! \! \{ a, u, b \}  \! \! \} := \{ \! \! \{ s(a), u , s (b) \}  \! \! \}_E, ~~ \{ \! \! \{ u, a, b \}  \! \! \} := \{ \! \! \{ u, s(a), s(b) \}  \! \! \}_E,
\end{align*}
for $a, b \in A$ and $u \in M$. Then it is easy to see that the above maps make $M$ into a representation of the given associative-Yamaguti algebra $(A, ~ \! \cdot ~ \! , \{ ~, ~ , ~ \}, \{ \! \! \{ ~, ~, ~ \} \! \! \})$, called the {\em induced representation}. Further, this representation doesn't depend on the choice of the section $s$. 

Let $(E, ~\! \cdot_E ~ \!, \{ ~, ~, ~ \}_E, \{ \! \! \{ ~, ~, ~ \} \! \! \}_E) $ and $(E', ~\! \cdot_{E'} ~ \!, \{ ~, ~, ~ \}_{E'}, \{ \! \! \{ ~, ~, ~ \} \! \! \}_{E'}) $ be two abelian extensions of the associative-Yamaguti algebra $(A, ~ \! \cdot ~ \! , \{ ~, ~ , ~ \}, \{ \! \! \{ ~, ~, ~ \} \! \! \})$ by the vector space $M$. They are said to be {\bf isomorphic} if there exists an isomorphism $\Phi : E \rightarrow E'$ of associative-Yamaguti algebras making the following diagram commutative:
\begin{align}\label{iso-abel}
    \xymatrix{
    0 \ar[r] & M \ar@{=}[d] \ar[r]^i & E \ar[r]^p \ar[d]^\Phi & A \ar[r] \ar@{=}[d] & 0 \\
    0 \ar[r] & M \ar[r]_{i'} & E' \ar[r]_{p'} & A \ar[r] & 0.
    }
\end{align}
Note that the induced representations on $M$ corresponding to isomorphic abelian extensions are the same.

Next, let $(M, ~ \! \cdot ~ \! , \{ ~, ~ , ~ \}, \{ \! \! \{ ~, ~, ~ \} \! \! \} )$ be any given representation of the associative-Yamaguti algebra $(A, ~ \! \cdot ~ \! , \{ ~, ~ , ~ \}, \{ \! \! \{ ~, ~, ~ \} \! \! \})$. We denote by $\mathrm{Ext} (A, M)$ the set of all isomorphism classes of abelian extensions of the associative-Yamaguti algebra $(A, ~ \! \cdot ~ \! , \{ ~, ~ , ~ \}, \{ \! \! \{ ~, ~, ~ \} \! \! \})$ by the vector space $M$ for which the induced representation on $M$ coincides with the prescribed one. Then we have the following result.

\begin{thm}\label{thm-abelian}
    Let $(A, ~ \! \cdot ~ \! , \{ ~, ~ , ~ \}, \{ \! \! \{ ~, ~, ~ \} \! \! \})$ be an associative-Yamaguti algebra and $(M, ~ \! \cdot ~ \! , \{ ~, ~ , ~ \}, \{ \! \! \{ ~, ~, ~ \} \! \! \} )$ be a given representation of it. Then there is a bijection
    %\begin{align*}
        $\mathrm{Ext} (A, M) \cong  \mathcal{H}^{(2, 3)} (A, M).$
    %\end{align*}
\end{thm}

\begin{proof}
    Let $(E, ~\! \cdot_E ~ \!, \{ ~, ~, ~ \}_E, \{ \! \! \{ ~, ~, ~ \} \! \! \}_E)$ be an abelian extension as of (\ref{abel-ext-diag}). For any section $s : A \rightarrow E$, we define linear maps $\mu : A^{\otimes 2} \rightarrow M$ and $F, G : A^{\otimes 3} \rightarrow M$ by
    \begin{align*}
        \mu (a, b) :=~& s (a) \cdot_E s(b) -  s (a \cdot b),\\
        F (a, b, c ) :=~& \{ s(a), s(b), s(c) \}_E - s ( \{ a, b, c \}),\\
        G (a, b, c ) :=~& \{ \! \! \{ s(a), s(b), s(c) \} \! \! \}_E - s ( \{ \! \! \{  a, b, c \} \! \! \}),
    \end{align*}
    for $a, b, c \in A$. Then it is not hard to see that the triple $(\mu, F, G)$ is a $(2, 3)$-cocycle. See the reference \cite{zhang-li} for a similar result in the context of Lie-Yamaguti algebras. Suppose $\widetilde{s} : A \rightarrow E$ is any other section with the corresponding $(2, 3)$-cocycle $(\widetilde{\mu}, \widetilde{F}, \widetilde{G})$. We set a linear map $\varphi : A \rightarrow M$ defined by $\varphi = s - \widetilde{s}$. Then one may show that
    \begin{align*}
        (\mu, F, G) - (\widetilde{\mu}, \widetilde{F}, \widetilde{G}) = (\mu_\varphi, F_\varphi, G_\varphi) \text{~ is a } (2,3)\text{-coboundary.}
    \end{align*}
    Hence, the $(2, 3)$-cocycles $(\mu, F, G)$ and $(\widetilde{\mu}, \widetilde{F}, \widetilde{G})$ correspond to the same element in $\mathcal{H}^{2,3} (A, M)$.

    Next, let $(E, ~\! \cdot_E ~ \!, \{ ~, ~, ~ \}_E, \{ \! \! \{ ~, ~, ~ \} \! \! \}_E)$ and $(E', ~\! \cdot_{E'} ~ \!, \{ ~, ~, ~ \}_{E'}, \{ \! \! \{ ~, ~, ~ \} \! \! \}_{E'})$ be two isomorphic abelian extensions as of (\ref{iso-abel}). For any section $s : A \rightarrow E$ (of the map $p$), the map $s'= \Phi \circ s : A \rightarrow E'$ is a section of the map $p'$. Therefore, if $(\mu', F', G')$ is the $(2, 3)$-cocycle corresponding to the abelian extension $(E', ~\! \cdot_{E'} ~ \!, \{ ~, ~, ~ \}_{E'}, \{ \! \! \{ ~, ~, ~ \} \! \! \}_{E'})$ with the section $s'= \Phi \circ s$, then we have
    \begin{align*}
        \mu'(a, b) = (\Phi \circ s) (a) \cdot_{E'} (\Phi \circ s) (b) - (\Phi \circ s) (a \cdot b) = \Phi ( \underbrace{s (a) \cdot_E s(b) -  s (a \cdot b)}_{ =~ \mu (a, b) ~\in M}   ) = \mu (a, b). 
        \end{align*}
    Similarly, $F'(a, b, c ) = F(a, b, c) $ and $G'(a, b, c ) = G(a, b, c ),$ for all $a, b, c \in A$. Therefore, we get that $(\mu, F, G) = (\mu', F', G')$. This ensures the existence of a map $\Psi: \mathrm{Ext} (A, M) \rightarrow \mathcal{H}^{(2,3)} (A, M).$

    \medskip

    On the other hand, let $(\mu, F, G) \in \mathcal{Z}^{(2, 3)} (A, M)$ be any $(2, 3)$-cocycle of the associative-Yamaguti algebra $(A, ~\! \cdot ~\! , \{ ~, ~, ~\} , \{ \! \! \{ ~, ~, ~\} \! \! \})$ with coefficients in the given representation. Then by Proposition \ref{twisted semid}, we obtain the associative-Yamaguti algebra $ A \oplus_{(\mu, F, G)} M = (A \oplus M, ~ \! \cdot_{\ltimes_\mu} ~ \! , \{ ~, ~, ~ \}_{\ltimes_F}, \{ \! \! \{ ~, ~, ~ \} \! \! \}_{\ltimes_G})$. It is easy to see that $0 \rightarrow M \xrightarrow{i}  A \oplus_{(\mu, F, G)} M \xrightarrow{p} A \rightarrow 0$ is an abelian extension, where $i$ is the obvious inclusion map and $p$ is the projection map. Further, the induced representation on $M$ coincides with the prescribed one. Next, suppose $(\mu, F, G)$ and $(\mu', F', G')$ are two $(2, 3)$-cocycles differ by a $(2, 3)$-coboundary, i.e.
    \begin{align*}
        (\mu, F, G) - (\mu', F', G') = (\mu_f, F_f, G_f),
    \end{align*}
    for some linear map $f: A \rightarrow M$. We define a linear isomorphism $\Phi : A \oplus M \rightarrow A \oplus M$ by  
    \begin{align}\label{phi-map}
        \Phi (a, u) = (a , u + f (a)), \text{ for } (a, u) \in A \oplus M.
    \end{align}
    Then we have
    \begin{align*}
         \Phi \big(  (a, u) \cdot_{\ltimes_\mu} (b, v) \big) \stackrel{(\ref{twisted-mu}) (\ref{phi-map})}{=}~& \big(  a \cdot b ~ \!, ~ \! a \cdot v + u \cdot b + \mu (a, b) + f (a \cdot b)  \big) \\
         =~& \big(  a \cdot b  ~ \!, ~ \! a \cdot v + a \cdot f(b) + u \cdot b + f(a) \cdot b + \mu' (a, b) \big) \quad (\because ~ \mu- \mu' = \mu_f)\\         
         =~& (a, u +f (a)) \cdot_{\ltimes_{\mu'}} (b , v + f(b)) = \Phi (a, u) \cdot_{\ltimes_{\mu'}} \Phi (b, v),
    \end{align*}
    for any $(a, u), (b, v) \in A \oplus M$. In the same way, one may show that
    \begin{align*}
        \Phi \big( \{ (a, u), (b, v), (c, w ) \}_{\ltimes_F} \big) =~& \{ \Phi (a, u), \Phi (b, v), \Phi (c, w ) \}_{\ltimes_{F'}}, \\
        \Phi \big( \{ \! \! \{  (a, u), (b, v), (c, w ) \} \! \!  \}_{\ltimes_G} \big) =~& \{ \! \! \{  \Phi (a, u), \Phi (b, v), \Phi (c, w ) \} \! \! \}_{\ltimes_{G'}}.
    \end{align*}
    Hence $\Phi : A \oplus_{(\mu, F, G)} M \rightarrow A \oplus_{(\mu', F', G')} M$ is an isomorphism of abelian extensions. As a result, we obtain a map $\Upsilon : \mathcal{H}^{(2,3)} (A, M) \rightarrow \mathrm{Ext} (A, M)$. Finally, it is easy to see that the maps $\Psi$ and $\Upsilon$ are inverses to each other. This completes the proof.
\end{proof}

\section{Yamaguti multiplications, dendriform-Yamaguti algebras and relative Rota-Baxter operators}\label{sec6}

In this section, we introduce Yamaguti multiplications on any nonsymmetric operad as a generalization of associative-Yamaguti algebras. Subsequently, we define and study dendriform-Yamaguti algebras. We end this section by considering relative Rota-Baxter operators on associative-Yamaguti algebras and their relations with dendriform-Yamaguti algebras.

\begin{defn}
    A {\bf nonsymmetric operad} (or a {\bf non-$\Sigma$ operad}) is a triple $\mathcal{P} = (\{ \mathcal{P} (n) \}_{n \geq 1}, \circ, {\bf 1})$ consisting of a collection $\{ \mathcal{P} (n) \}_{n \geq 1}$ of vector spaces equipped with linear maps (called {\em partial compositions})
    \begin{align*}
        \circ_i : \mathcal{P} (m) \otimes \mathcal{P}(n) \rightarrow \mathcal{P} (m+n-1) \quad (\text{for } 1 \leq i \leq m)
    \end{align*}
    such that for any $f \in \mathcal{P}(m)$, $g \in \mathcal{P}(n)$ and $h \in \mathcal{P} (k),$
    \begin{align*}
        f \circ_i (g \circ_j h) =~& (f \circ_i g) \circ_{i+j-1} h, ~~  \text{ for } 1 \leq i \leq m, 1 \leq j \leq n, \\
        (f \circ_i  g) \circ_{j+n-1} h =~& (f \circ_j h) \circ_i  g, ~~ \text{ for } 1 \leq i < j \leq m,
    \end{align*}
    and there is a distinguished element ${\bf 1} \in \mathcal{P}(1)$ that satisfies $f \circ_i {\bf 1} = {\bf 1} \circ_1 f$, for all $f \in \mathcal{P} (m)$ and $1 \leq i \leq m.$
\end{defn}

In this paper, we will mostly concentrate on the following two nonsymmetric operads.

\begin{exam}
    Let $A$ be any vector space. Then the endomorphism operad $\mathrm{End}_A = ( \{ \mathrm{End}_A (n) \}_{n \geq 1}, \circ, \mathrm{Id}_A )$ associated to $A$ is given by $\mathrm{End}_A (n) = \mathrm{Hom} (A^{\otimes n} , A)$ for $n \geq 1$, with the partial compositions
    \begin{align*}
         (f \circ_i g) (a_1, \ldots, a_{m+n-1} ) = f (a_1, \ldots, a_{i-1} ,  g (a_i, \ldots, a_{i+n-1}), a_{i+n}, \ldots, a_{m+n-1}),
    \end{align*}
    for $f \in \mathrm{End}_A (m)$, $g \in \mathrm{End}_A (n)$, $1 \leq i \leq m$ and $a_1, \ldots, a_{m+n-1} \in A$. Here $\mathrm{Id}_A \in \mathrm{End}_A(1) = \mathrm{Hom}(A,A)$ is the identity map.
\end{exam}

\begin{exam} Let $C_n$ be the set of first $n$ natural numbers. For convenience and notational purpose, we write $C_n = \{   [1], [2], \ldots , [n] \}$. Given any vector space $A$, we set $\mathrm{Dend}_A (n) = \mathrm{Hom} ( {\bf k} [C_n] \otimes A^{\otimes n}, A)$ for $n \geq 1$. For any $f \in \mathrm{Dend}_A (m)$, $g \in \mathrm{Dend}_A (n)$ and $1 \leq i \leq m$, one define an element $f \circ_i g \in \mathrm{Dend}_A (m+n-1)$ by 
\begin{align*}
    &(f \circ_i g) ([r]; a_1, \ldots, a_{m+n-1}) \\
   &~= \begin{cases}
        f ({\scriptstyle [r]} ;  a_1, \ldots, a_{i-1},  g ( {\scriptstyle [1] + \cdots + [n]} ; a_i, \ldots, a_{i+n-1}) , a_{i+n} , \ldots, a_{m+n-1}) & \text{ if } 1 \leq r \leq i-1, \\
        f ({\scriptstyle [i]} ;  a_1, \ldots, a_{i-1},  g ( {\scriptstyle [r-i+1]} ; a_i, \ldots, a_{i+n-1}) , a_{i+n} , \ldots, a_{m+n-1}) & \text{ if } i \leq r \leq i+n-1, \\
        f ({\scriptstyle [r-n+1]} ;  a_1, \ldots, a_{i-1},  g ( {\scriptstyle [1] + \cdots + [n]} ; a_i, \ldots, a_{i+n-1}) , a_{i+n} , \ldots, a_{m+n-1}) & \text{ if } i+n \leq r \leq m+n-1,
    \end{cases}
\end{align*}
where $[r] \in C_{m+n-1}$ and $a_1, \ldots, a_{m+n-1} \in A$. We also consider the element $\mathrm{Id}_A \in \mathrm{Dend}_A (1)$ which is being defined by $\mathrm{Id}_A ( [1]; a) = a$, for all $a \in A$. Then $\mathrm{Dend}_A = (     \{ \mathrm{Dend}_A (n) \}_{n \geq 1}, \circ, \mathrm{Id}_A)$ is a nonsymmetric operad \cite{das-dend}.
\end{exam}

Let $\mathcal{P} = ( \{ \mathcal{P} (n) \}_{n \geq 1}, \circ, {\bf 1})$ be any nonsymmetric operad. Then an element $\pi \in \mathcal{P}(2)$ is said to be a {\bf multiplication} \cite{gers-voro} on $\mathcal{P}$ if $\pi \circ_1 \pi = \pi \circ_2 \pi$. It follows that a multiplication on the endomorphism operad $\mathrm{End}_A$ is a linear map $\pi : A \otimes A \rightarrow A$, $a \otimes b \mapsto a \cdot b$ satisfying $(a \cdot b) \cdot c  = a \cdot (b \cdot c)$, for all $a, b, c \in A$. Hence, multiplications on $\mathrm{End}_A$ correspond to associative algebra structures on $A$. On the other hand, it is easy to see that a multiplication $\pi$ on the nonsymmetric operad $\mathrm{Dend}_A$ is equivalent to having two linear maps $\prec, \succ : A \otimes A \rightarrow A$ satisfying
    \begin{align}
        (a \prec b) \prec c =~& a \prec (b \prec c + b \succ c), \label{dend1} \\
        (a \succ b) \prec c =~& a \succ (b \prec c),\\
        (a \prec b + a \succ b) \succ c=~& a \succ (b \succ c), \label{dend3}
    \end{align}
    for all $a, b, c \in A$. Such correspondence is given by
    \begin{align*}
        \pi ~ ~  \Longleftrightarrow ~ ~ a \prec b = \pi ([1] ; a, b) ~ \text{ and } ~ a \succ b = \pi ([2]; a, b).
    \end{align*}
    A triple $(A, \prec, \succ)$ of a vector space $A$ with two linear operations $\prec, \succ : A \otimes A \rightarrow A$ satisfying the identities (\ref{dend1})-(\ref{dend3}) is called a {\bf dendriform algebra} \cite{loday}. Thus, multiplications on $\mathrm{Dend}_A$ correspond to dendriform algebra structures on $A$.

\begin{defn}\label{defn-ym}
    Let $\mathcal{P}$ be a nonsymmetric operad. A {\bf Yamaguti multiplication} on $\mathcal{P}$ is a triple $(\pi, \theta, \vartheta)$ consisting of elements $\pi \in \mathcal{P}(2)$ and $\theta, \vartheta \in \mathcal{P}(3)$ satisfying the following conditions:
\begin{align*}
     \pi \circ_1 \pi - \pi \circ_2 \pi + \theta - \vartheta = 0, \qquad 
     \begin{cases}
         \theta \circ_1 \pi = \theta \circ_2 \pi, \\
         \theta \circ_3 \pi = \pi \circ_1 \theta, \\
         \vartheta \circ_1 \pi = \pi \circ_2 \vartheta,\\
         \vartheta \circ_2 \pi = \vartheta \circ_3 \pi,\\
         \pi \circ_2 \theta = \pi \circ_1 \vartheta,
     \end{cases} \qquad 
     \begin{cases}
         \theta \circ_1 \theta = \theta \circ_2 \vartheta = \theta \circ_3 \theta,\\
         \theta \circ_2 \theta = \theta \circ_1 \vartheta,\\
         \vartheta \circ_1 \vartheta = \vartheta \circ_2 \theta = \vartheta \circ_3 \vartheta,\\
         \vartheta \circ_2 \vartheta = \vartheta \circ_3 \theta,\\
         \theta \circ_3 \vartheta = \vartheta \circ_1 \theta.
     \end{cases}
\end{align*}
\end{defn}

\begin{exam}
    Let $\pi$ be a multiplication on $\mathcal{P}$. Define elements $\theta, \vartheta \in \mathcal{P}(3)$ by $\theta = \vartheta = \pi \circ_1 \pi = \pi \circ_2 \pi$. Then $(\pi, \theta, \vartheta)$ is a Yamaguti multiplication on $\mathcal{P}$.
\end{exam}

\begin{thm}\label{thm-ym-assy}
    Let $A$ be a vector space. Then there is a one-one correspondence between associative-Yamaguti algebra structures on $A$ and Yamaguti multiplications on the endomorphism operad $\mathrm{End}_A$.
\end{thm}

\begin{proof}
    Let $(A, ~ \! \cdot  ~ \!, \{ ~, ~ , ~ \}, \{ \! \! \{ ~, ~, ~ \} \! \! \})$ be an associative-Yamaguti algebra structure on the vector space $A$. We define elements $\pi \in \mathrm{End}_A (2)$ and $\theta, \vartheta \in \mathrm{End}_A (3)$ by 
    \begin{align}\label{aya-aym}
        \pi (a, b) = a \cdot b, \qquad \theta (a, b, c) = \{ a, b, c \}, \qquad \vartheta (a, b, c ) = \{ \! \! \{ a, b, c \} \! \! \},
    \end{align}
    for $a, b, c \in A$. Then it is easy to see that $(\pi, \theta, \vartheta)$ is a Yamaguti multiplication on $\mathrm{End}_A$. Using (\ref{aya-aym}), one may conversely define an associative-Yamaguti algebra structure on $A$ from a Yamaguti multiplication on $\mathrm{End}_A$.
\end{proof}

Motivated by the above result, one may think that a dendriform-Yamaguti algebra structure on a vector space $A$ should be a Yamaguti multiplication on the nonsymmetric operad $\mathrm{Dend}_A$. Explicit form is given in the following definition.

\begin{defn} \label{defn-dendy} A {\bf dendriform-Yamaguti algebra} is a tuple 
\begin{align}\label{dy-algebra}
    (D, \prec, \succ, \{ ~, ~, ~ \}_{[1]},\{ ~, ~, ~ \}_{[2]}, \{ ~, ~, ~ \}_{[3]}, \{ \! \! \{ ~, ~ , ~ \} \! \! \}_{[1]}, \{ \! \! \{ ~, ~ , ~ \} \! \! \}_{[2]}, \{ \! \! \{ ~, ~ , ~ \} \! \! \}_{[3]} )
\end{align}
of a vector space $D$ equipped with linear operations $\prec, \succ : D \otimes D \rightarrow D$ and 
\begin{align*}
     \{ ~, ~, ~ \}_{[1]} ~ , ~\{ ~, ~, ~ \}_{[2]} ~,~ \{ ~, ~, ~ \}_{[3]} ~,~ \{ \! \! \{ ~, ~ , ~ \} \! \! \}_{[1]}~,~ \{ \! \! \{ ~, ~ , ~ \} \! \! \}_{[2]} ~,~ \{ \! \! \{ ~, ~ , ~ \} \! \! \}_{[3]} : D \otimes D \otimes D \rightarrow D
\end{align*}
such that for all $a, b, c, d, e \in D$, the following list of identities are hold:
%\begin{itemize}
 %   \item[(DY1A)] \qquad $(a \prec b) \prec c - a \prec (b \prec c + b \succ c) + \{ a, b, c \}_{[1]} - \{ \! \! \{ a, b, c \} \! \! \}_{[1]} = 0,$ 
 %   \medskip    
%    \item[(DY2A)] \qquad $(a \succ b) \prec c - a \succ (b \prec c) + \{ a, b, c \}_{[2]} - \{ \! \! \{ a, b, c \} \! \! \}_{[2]} = 0,$
%\end{itemize}
\begin{align*}
   &\mathrm{(DY1A)} \qquad (a \prec b) \prec c - a \prec (b \prec c + b \succ c) + \{ a, b, c \}_{[1]} - \{ \! \! \{ a, b, c \} \! \! \}_{[1]} = 0,\\
   &\mathrm{(DY1B)} \qquad (a \succ b) \prec c - a \succ (b \prec c) + \{ a, b, c \}_{[2]} - \{ \! \! \{ a, b, c \} \! \! \}_{[2]} = 0,\\
   &\mathrm{(DY1C)} \qquad (a \prec b + a \succ b) \succ c - a \succ (b \succ c) + \{ a, b, c \}_{[3]} - \{ \! \! \{ a, b, c \} \! \! \}_{[3]} = 0,
\end{align*}
\begin{align*}
    &\mathrm{(DY2A)} \qquad  \{ a \prec b, c, d \}_{[1]} = \{ a, b \prec c + b \succ c, d \}_{[1]},\\
    &\mathrm{(DY2B)} \qquad \{ a \succ b, c, d \}_{[1]} = \{ a, b \prec c , d \}_{[2]}, \\
    &\mathrm{(DY2C)} \qquad \{ a \prec b + a \succ b , c, d \}_{[2]} = \{ a, b \succ c , d \}_{[2]},\\
    &\mathrm{(DY2D)} \qquad \{ a \prec b + a \succ b , c, d \}_{[3]} = \{ a, b \prec c + b \succ c , d \}_{[3]},
\end{align*}
\begin{align*}
     &\mathrm{(DY3A)} \qquad \{ a, b, c \prec d + c \succ d \}_{[1]} = \{ a, b, c \}_{[1]} \prec d,\\
     &\mathrm{(DY3B)} \qquad \{ a, b, c \prec d + c \succ d \}_{[2]} = \{ a, b, c \}_{[2]} \prec d,\\
     &\mathrm{(DY3C)} \qquad \{ a, b, c \prec d \}_{[3]} = \{ a, b, c \}_{[3]} \prec d,\\
     &\mathrm{(DY3D)} \qquad \{ a, b, c \succ d \}_{[3]} = \{ a, b, c \}_\mathrm{Tot} \succ d,
\end{align*}
\begin{align*}
    &\mathrm{(DY4A)} \qquad  \{ \! \! \{ a \prec b, c, d \} \! \! \}_{[1]} = a \prec \{ \! \! \{ b, c, d \} \! \! \}_\mathrm{Tot},\\
    &\mathrm{(DY4B)} \qquad  \{ \! \! \{ a \succ b, c, d \} \! \! \}_{[1]} = a \succ \{ \! \! \{ b, c, d \} \! \! \}_{[1]},\\
     &\mathrm{(DY4C)} \qquad  \{ \! \! \{ a \prec b + a \succ b , c, d \} \! \! \}_{[2]} = a \succ \{ \! \! \{ b, c, d \} \! \! \}_{[2]},\\
     &\mathrm{(DY4D)} \qquad   \{ \! \! \{ a \prec b + a \succ b , c, d \} \! \! \}_{[3]} = a \succ \{ \! \! \{ b, c, d \} \! \! \}_{[3]},
\end{align*}
\begin{align*}
    &\mathrm{(DY5A)} \qquad  \{ \! \! \{ a, b \prec c + b \succ c , d \} \! \! \}_{[1]} = \{ \! \! \{ a, b, c \prec d + c \succ d \} \! \! \}_{[1]}, \\
    &\mathrm{(DY5B)} \qquad  \{ \! \! \{ a, b \prec c , d \} \! \! \}_{[2]} = \{ \! \! \{ a, b, c \prec d + c \succ d \} \! \! \}_{[2]},\\
     &\mathrm{(DY5C)} \qquad  \{ \! \! \{ a, b \succ c , d \} \! \! \}_{[2]} = \{ \! \! \{ a, b, c \prec d \} \! \! \}_{[3]},\\
     &\mathrm{(DY5C)} \qquad  \{ \! \! \{ a, b \prec c + b \succ c , d \} \! \! \}_{[3]} = \{ \! \! \{ a, b, c \succ d \} \! \! \}_{[3]},
\end{align*}
\begin{align*}
     &\mathrm{(DY6A)} \qquad  a \prec \{ b, c, d   \}_\mathrm{Tot} = \{ \! \! \{ a, b, c    \} \! \! \}_{[1]} \prec d, \\
    &\mathrm{(DY6B)} \qquad  a \succ \{ b, c, d   \}_{[1]} = \{ \! \! \{ a, b, c     \} \! \! \}_{[2]} \prec d, \\
    &\mathrm{(DY6C)} \qquad  a \succ \{ b, c, d   \}_{[2]} = \{ \! \! \{ a, b, c     \} \! \! \}_{[3]} \prec d, \\
    &\mathrm{(DY6D)} \qquad  a \succ \{ b, c, d   \}_{[3]} = \{ \! \! \{  a, b, c   \} \! \! \}_\mathrm{Tot} \succ d,
\end{align*}
\begin{align*}
     &\mathrm{(DY7A)} \qquad  \{    \{ a, b, c    \}_{[1]}, d, e    \}_{[1]}  = \{a, \{ \! \! \{ b, c , d    \} \! \! \}_\mathrm{Tot} , e \}_{[1]} = \{ a, b, \{ c, d, e \}_\mathrm{Tot} \}_{[1]},\\
    &\mathrm{(DY7B)} \qquad   \{    \{ a, b, c \}_{[2]}, d, e    \}_{[1]}  = \{a, \{ \! \! \{ b, c , d    \} \! \! \}_{[1]}, e \}_{[2]} = \{ a, b, \{ c, d, e \}_\mathrm{Tot} \}_{[2]}, \\
     &\mathrm{(DY7C)} \qquad   \{    \{ a, b, c    \}_{[3]}, d, e    \}_{[1]}  = \{a, \{ \! \! \{ b, c , d    \} \! \! \}_{[2]}, e \}_{[2]} = \{ a, b, \{ c, d, e \}_{[1]} \}_{[3]}, \\
      &\mathrm{(DY7D)} \qquad   \{    \{ a, b, c    \}_\mathrm{Tot}, d, e    \}_{[2]}  = \{a, \{ \! \! \{ b, c , d    \} \! \! \}_{[3]}, e \}_{[2]} = \{ a, b, \{ c, d, e \}_{[2]} \}_{[3]}, \\
      &\mathrm{(DY7E)} \qquad    \{    \{ a, b, c    \}_\mathrm{Tot} , d, e    \}_{[3]}  = \{a, \{ \! \! \{ b, c , d    \} \! \! \}_\mathrm{Tot} , e \}_{[3]} = \{ a, b, \{ c, d, e \}_{[3]} \}_{[3]},
\end{align*}
\begin{align*}
    &\mathrm{(DY8A)} \qquad  \{ a, \{ b, c , d\}_\mathrm{Tot} , e \}_{[1]} = \{   \{ \! \! \{ a, b, c \} \! \! \}_{[1]}, d, e \}_{[1]}, \\
     &\mathrm{(DY8B)} \qquad   \{ a, \{ b, c , d\}_{[1]} , e \}_{[2]} = \{   \{ \! \! \{ a, b, c \} \! \! \}_{[2]}, d, e \}_{[1]}, \\
     &\mathrm{(DY8C)} \qquad   \{ a, \{ b, c , d\}_{[2]} , e \}_{[2]} = \{   \{ \! \! \{ a, b, c \} \! \! \}_{[3]}, d, e \}_{[1]},\\
      &\mathrm{(DY8D)} \qquad   \{ a, \{ b, c , d\}_{[3]} , e \}_{[2]} = \{   \{ \! \! \{ a, b, c \} \! \! \}_\mathrm{Tot}, d, e \}_{[2]},\\
       &\mathrm{(DY8E)} \qquad   \{ a, \{ b, c , d\}_\mathrm{Tot} , e \}_{[3]} = \{   \{ \! \! \{ a, b, c \} \! \! \}_\mathrm{Tot}, d, e \}_{[3]},
\end{align*}
\begin{align*}
   &\mathrm{(DY9A)} \qquad   \{ \! \! \{    \{ \! \! \{ a, b, c    \} \! \! \}_{[1]}, d, e   \} \! \! \}_{[1]}  = \{ \! \! \{a,  \{ b, c , d    \}_\mathrm{Tot}, e \} \! \! \}_{[1]} = \{ \! \! \{ a, b, \{ \! \!  \{ c, d, e  \} \! \!  \}_\mathrm{Tot}  \} \! \!  \}_{[1]}, \\
   &\mathrm{(DY9B)} \qquad   \{ \! \! \{    \{ \! \! \{ a, b, c    \} \! \! \}_{[2]}, d, e   \} \! \! \}_{[1]}  = \{ \! \! \{a,  \{ b, c , d    \}_{[1]}, e \} \! \! \}_{[2]} = \{ \! \! \{ a, b, \{ \! \!  \{ c, d, e  \} \! \!  \}_\mathrm{Tot}  \} \! \!  \}_{[2]}, \\
   &\mathrm{(DY9C)} \qquad   \{ \! \! \{    \{ \! \! \{ a, b, c    \} \! \! \}_{[3]}, d, e   \} \! \! \}_{[1]}  = \{ \! \! \{a,  \{ b, c , d    \}_{[2]}, e \} \! \! \}_{[2]} = \{ \! \! \{ a, b, \{ \! \!  \{ c, d, e  \} \! \!  \}_{[1]}  \} \! \!  \}_{[3]}, \\
    &\mathrm{(DY9D)} \qquad    \{ \! \! \{    \{ \! \! \{ a, b, c    \} \! \! \}_\mathrm{Tot}, d, e   \} \! \! \}_{[2]}  = \{ \! \! \{a,  \{ b, c , d    \}_{[3]}, e \} \! \! \}_{[2]} = \{ \! \! \{ a, b, \{ \! \!  \{ c, d, e  \} \! \!  \}_{[2]}  \} \! \!  \}_{[3]}, \\
    &\mathrm{(DY9E)} \qquad    \{ \! \! \{    \{ \! \! \{ a, b, c    \} \! \! \}_\mathrm{Tot} , d, e   \} \! \! \}_{[3]}  = \{ \! \! \{a,  \{ b, c , d    \}_\mathrm{Tot}, e \} \! \! \}_{[3]} = \{ \! \! \{ a, b, \{ \! \!  \{ c, d, e  \} \! \!  \}_{[3]}  \} \! \!  \}_{[3]},
\end{align*}
\begin{align*}
    &\mathrm{(DY10A)} \qquad   \{ \! \! \{ a, \{ \! \! \{ b, c, d \} \! \! \}_\mathrm{Tot} , e \} \! \! \}_{[1]} = \{ \! \! \{ a, b, \{ c, d, e \}_\mathrm{Tot}  \} \! \! \}_{[1]}, \\
    &\mathrm{(DY10B)} \qquad   \{ \! \! \{ a, \{ \! \! \{ b, c, d \} \! \! \}_{[1]}, e \} \! \! \}_{[2]} = \{ \! \! \{ a, b, \{ c, d, e \}_\mathrm{Tot}  \} \! \! \}_{[2]}, \\
     &\mathrm{(DY10C)} \qquad   \{ \! \! \{ a, \{ \! \! \{ b, c, d \} \! \! \}_{[2]}, e \} \! \! \}_{[2]} = \{ \! \! \{ a, b, \{ c, d, e \}_{[1]} \} \! \! \}_{[3]}, \\
       &\mathrm{(DY10D)} \qquad  \{ \! \! \{ a, \{ \! \! \{ b, c, d \} \! \! \}_{[3]}, e \} \! \! \}_{[2]} = \{ \! \! \{ a, b, \{ c, d, e \}_{[2]} \} \! \! \}_{[3]},\\
       &\mathrm{(DY10E)} \qquad  \{ \! \! \{ a, \{ \! \! \{ b, c, d \} \! \! \}_\mathrm{Tot} , e \} \! \! \}_{[3]} = \{ \! \! \{ a, b, \{ c, d, e \}_{[3]} \} \! \! \}_{[3]},
\end{align*}
\begin{align*}
     &\mathrm{(DY11A)} \qquad  \{ a, b, \{ \! \! \{ c, d, e \} \! \! \}_\mathrm{Tot}  \}_{[1]} = \{ \! \! \{     \{ a, b, c \}_{[1]} , d, e   \} \! \! \}_{[1]},\\
    &\mathrm{(DY11B)} \qquad  \{ a, b, \{ \! \! \{ c, d, e \} \! \! \}_\mathrm{Tot}  \}_{[2]} = \{ \! \! \{     \{ a, b, c \}_{[2]} , d, e   \} \! \! \}_{[1]}, \\
     &\mathrm{(DY11C)} \qquad  \{ a, b, \{ \! \! \{ c, d, e \} \! \! \}_{[1]}    \}_{[3]} = \{ \! \! \{     \{ a, b, c \}_{[3]} , d, e   \} \! \! \}_{[1]}, \\
     &\mathrm{(DY11D)} \qquad   \{ a, b, \{ \! \! \{ c, d, e \} \! \! \}_{[2]}   \}_{[3]} = \{ \! \! \{     \{ a, b, c \}_\mathrm{Tot} , d, e   \} \! \! \}_{[2]}, \\
     &\mathrm{(DY11E)} \qquad    \{ a, b, \{ \! \! \{ c, d, e \} \! \! \}_{[3]}    \}_{[3]} = \{ \! \! \{     \{ a, b, c \}_\mathrm{Tot} , d, e   \} \! \! \}_{[3]}.
\end{align*}
Here for any $a, b, c \in D$, we used the notations
\begin{align}\label{total-dend}
    \{ a, b, c \}_\mathrm{Tot} = \{ a, b, c \}_{[1]} + \{ a, b, c \}_{[2]} + \{ a, b, c \}_{[3]} ~ \text{ and } ~ \{ \! \! \{ a, b, c \} \! \! \}_\mathrm{Tot} = \{ \! \! \{ a, b, c \} \! \! \}_{[1]} + \{ \! \! \{ a, b, c \} \! \! \}_{[2]} + \{ \! \! \{ a, b, c \} \! \! \}_{[3]}.
\end{align}
\end{defn}

We may denote a dendriform-Yamaguti algebra as above simply by $D$ when the underlying operations are clear from the context. Let $D$ and $D'$ be two dendriform-Yamaguti algebras. A {\bf homomorphism} of dendriform-Yamaguti algebras from $D$ to $D'$ is a linear map $\psi : D \rightarrow D'$ that preserves the corresponding operations. 
%The collection of all dendriform-Yamaguti algebras and homomorphisms between them forms a category, denoted by {\bf DendY}.

\begin{exam}
    Let $(D, \prec, \succ)$ be any dendriform algebra. Then for any $a, b, c \in D$, we define
    \begin{align*}
        \{ a, b, c \}_{[1]} =~& \{ \! \! \{ a, b, c \} \! \! \}_{[1]} = (a \prec b ) \prec c = a \prec (b \prec c + b \succ c), \\
        \{ a, b, c \}_{[2]} =~& \{ \! \! \{ a, b, c \} \! \! \}_{[2]} = (a \succ b) \prec c = a \succ (b \prec c),\\
        \{ a, b, c \}_{[3]} =~& \{ \! \! \{ a, b, c \} \! \! \}_{[3]} = (a \prec b + a \succ b) \succ c = a \succ (b \succ c).
    \end{align*}
    Then $(D, \prec, \succ, \{ ~, ~, ~ \}_{[1]},\{ ~, ~, ~ \}_{[2]}, \{ ~, ~, ~ \}_{[3]}, \{ \! \! \{ ~, ~ , ~ \} \! \! \}_{[1]}, \{ \! \! \{ ~, ~ , ~ \} \! \! \}_{[2]}, \{ \! \! \{ ~, ~ , ~ \} \! \! \}_{[3]} )$ is a dendriform-Yamaguti algebra.
\end{exam}

\begin{exam}
    A {\bf dendriform triple system} is a tuple $(D, \{ ~, ~ , ~ \}_{[1]}, \{ ~, ~ , ~ \}_{[2]}, \{ ~, ~, ~ \}_{[3]})$ of a vector space $D$ with linear operations $\{ ~, ~ , ~ \}_{[1]}, \{ ~, ~ , ~ \}_{[2]}, \{ ~, ~, ~ \}_{[3]} : D \otimes D \otimes D \rightarrow D$ satisfying
    \begin{equation}
        \{ \{ a, b, c \}_{[1]} , d, e \}_{[1]} = \{ a, \{ b, c, d \}_\mathrm{Tot}, e \}_{[1]} = \{ a, b, \{ c, d, e \}_\mathrm{Tot} \}_{[1]},
    \end{equation}
     \begin{equation}
        \{ \{ a, b, c \}_{[2]} , d, e \}_{[1]} = \{ a, \{ b, c, d \}_{[1]}, e \}_{[2]} = \{ a, b, \{ c, d, e \}_\mathrm{Tot} \}_{[2]},
    \end{equation}
     \begin{equation}
        \{ \{ a, b, c \}_{[3]} , d, e \}_{[1]} = \{ a, \{ b, c, d \}_{[2]}, e \}_{[2]} = \{ a, b, \{ c, d, e \}_{[1]} \}_{[3]},
    \end{equation}
     \begin{equation}
        \{ \{ a, b, c \}_\mathrm{Tot} , d, e \}_{[2]} = \{ a, \{ b, c, d \}_{[3]}, e \}_{[2]} = \{ a, b, \{ c, d, e \}_{[2]} \}_{[3]},
    \end{equation}
     \begin{equation}
        \{ \{ a, b, c \}_\mathrm{Tot} , d, e \}_{[3]} = \{ a, \{ b, c, d \}_\mathrm{Tot}, e \}_{[3]} = \{ a, b, \{ c, d, e \}_{[3]} \}_{[3]},
    \end{equation} 
    for $a, b, c, d, e \in D$. Here $\{ a, b, c \}_\mathrm{Tot} = \{ a, b, c \}_{[1]} + \{ a, b, c \}_{[2]} + \{ a, b, c \}_{[3]}$. It follows that a dendriform triple system can be regarded as a dendriform-Yamaguti algebra with trivial $\prec$ and $\succ$.
\end{exam}

\begin{exam}
    Any associative-Yamaguti algebra $(A, ~ \! \cdot ~ \! , \{ ~, ~,~ \}, \{ \! \! \{ ~, ~, ~ \} \! \! \})$ can be regarded as a dendriform-Yamaguti algebra in which $a \prec b = a\cdot b$, $ a \succ b = 0$ and
    \begin{align*}
        \{ a, b, c \}_{[1]} = \{ a, b, c \}, \quad  \{ \! \! \{ a, b, c \} \! \! \}_{[1]} = \{ \! \! \{ a, b, c \} \! \! \}, \quad \{ a, b, c \}_{[2]} = \{ a, b, c \}_{[3]} = \{ \! \! \{ a, b, c \} \! \! \}_{[2]}= \{ \! \! \{ a, b, c \} \! \! \}_{[3]} = 0,
    \end{align*}
    for all $a, b, c \in D$. More generally, a dendriform-Yamaguti algebra (\ref{dy-algebra}) in which exactly one of $\prec$ or $\succ$ is nontrivial, exactly one of $\{ ~, ~ , ~ \}_{[1]}, \{ ~, ~ , ~ \}_{[2]}$ or $\{ ~, ~ , ~ \}_{[3]}$ is nontrivial, and exactly one of $\{ \! \! \{ ~, ~ , ~ \} \! \! \}_{[1]}, \{ \! \! \{ ~, ~ , ~ \} \! \! \}_{[2]}$ or $\{ \! \! \{ ~, ~ , ~ \} \! \! \}_{[3]}$ is nontrivial, is nothing but an associative-Yamaguti algebra.
\end{exam}

\begin{thm}\label{total-assy}
    Let $(D, \prec, \succ, \{ ~, ~, ~ \}_{[1]},\{ ~, ~, ~ \}_{[2]}, \{ ~, ~, ~ \}_{[3]}, \{ \! \! \{ ~, ~ , ~ \} \! \! \}_{[1]}, \{ \! \! \{ ~, ~ , ~ \} \! \! \}_{[2]}, \{ \! \! \{ ~, ~ , ~ \} \! \! \}_{[3]} )$ be a dendriform-Yamaguti algebra. Then the quadruple $(D, ~ \! \cdot_\mathrm{Tot} ~ \! , \{ ~, ~ , ~ \}_\mathrm{Tot} , \{ \! \! \{ ~, ~, ~ \} \! \! \}_\mathrm{Tot})$ is an associative-Yamaguti algebra, where 
    \begin{align*}
        a \cdot_\mathrm{Tot} b :=  a \prec b + a \succ b, \text{ for } a, b \in D,
    \end{align*}
    and the operations $\{ ~, ~ , ~ \}_\mathrm{Tot}$ and $\{ \! \! \{ ~, ~ , ~ \} \! \! \}_\mathrm{Tot}$ are given in (\ref{total-dend}). Further, if $\psi : D \rightarrow D'$ is a homomorphism of dendriform-Yamaguti algebras, then $\psi$ is also a homomorphism of the corresponding associative-Yamaguti algebras.
\end{thm}

\begin{proof}
    By additing the identities (DY1A)-(DY1C), one simply obtain 
    \begin{align*}
        (a \cdot_\mathrm{Tot} b ) \cdot_\mathrm{Tot} c - a \cdot_\mathrm{Tot} (b \cdot_\mathrm{Tot} c) + \{ a, b, c \}_\mathrm{Tot} - \{ \! \! \{ a, b, c \} \! \! \}_\mathrm{Tot} = 0.
    \end{align*}
    Thus, the identity (\ref{ay1}) of an associative-Yamaguti algebra follows. Similarly, by adding the left-hand sides of the identities (DY2A)-(DY2D) and also the right-hand sides of the same identities, one gets
    \begin{align*}
        \{ a \cdot_\mathrm{Tot} b , c, d \}_\mathrm{Tot} = \{ a, b \cdot_\mathrm{Tot} c, d \}_\mathrm{Tot},
    \end{align*}
    which is the identity (\ref{ay2}). The other identities of an associative-Yamaguti algebra similarly follow by applying
    \begin{center}
        (DYiA) + (DYiB) + (DYiC) + (DYiD) $\Rightarrow$ (AYi), for $i=3,4,5,6,$
    \end{center}
    \begin{center}
        (DYjA) + (DYjB) + (DYjC) + (DYjD) + (DYjE) $\Rightarrow$ (AYj), for $j=7,8,9,10,11$.
    \end{center}
    This proves the first part of the theorem. For the second part, we observe that
    \begin{align*}
        \psi (a \cdot_\mathrm{Tot} b) = \psi (a \prec b + a \succ b) = \psi (a) \prec' \psi (b) + \psi (a) \succ' \psi (b) = \psi (a) \cdot'_\mathrm{Tot} \psi (b), 
        \end{align*}
        \begin{align*}
        \psi (\{ a, b, c \}_\mathrm{Tot} ) =~& \psi (\{ a, b, c \}_{[1]} + \{ a, b, c \}_{[2]} + \{ a, b, c \}_{[3]}) \\
        =~& \{ \psi (a) , \psi (b), \psi (c) \}'_{[1]} + \{ \psi (a), \psi (b), \psi (c) \}'_{[2]} + \{ \psi (a), \psi (b), \psi (c) \}'_{[3]} \\=~& \{ \psi (a) , \psi (b) , \psi (c) \}'_\mathrm{Tot},
    \end{align*}
    and similarly, $\psi (\{ \! \! \{ a, b, c \} \! \! \}_\mathrm{Tot} ) = \{ \! \! \{ \psi (a), \psi (b) , \psi (c) \} \! \! \}'_\mathrm{Tot}$, for all $a, b, c \in D.$ Hence the map $\psi : D \rightarrow D'$ is a homomorphism of the corresponding associative-Yamaguti algebras.
\end{proof}

We will now consider relative Rota-Baxter operators on associative-Yamaguti algebras and find their relationships with dendriform-Yamaguti algebras.

\begin{defn}
   Let $(A, ~ \! \cdot ~ \! , \{ ~, ~ , ~ \} , \{ \! \! \{ ~, ~, ~ \} \! \! \})$ be an associative-Yamaguti algebra and  $(M, ~ \! \cdot ~ \! , \{ ~, ~ , ~ \} , \{ \! \! \{ ~, ~, ~ \} \! \! \})$  be a representation of it. Then a linear map $R: M \rightarrow A$ is said to be a {\bf relative Rota-Baxter operator} (also called an {\bf $\mathcal{O}$-operator}) if
\begin{align*}
    R(u ) \cdot R(v) =~& R (R(u) \cdot v + u \cdot R(v)),\\
    \{ R(u), R(v) , R(w) \} =~& R \big( \{ R(u), R(v), w \} + \{ R(u), v, R(w) \} + \{ u, R(v) , R(w) \} \big),\\
    \{ \! \! \{  R(u), R(v)  , R(w)  \} \! \! \} =~& R \big(    \{ \! \! \{ R(u) , R(v) , w \} \! \! \} + \{ \! \! \{ R(u) , v , R(w) \} \! \! \} + \{ \! \! \{ u , R(v) , R(w) \} \! \! \} \big), \text{ for all } u, v, w \in M.
\end{align*}  
\end{defn}
It is easy to see that an invertible linear map $f: A \rightarrow M$ is a derivation on $A$ with values in $M$ if and only if the inverse $f^{-1}: M \rightarrow A$ is a relative Rota-Baxter operator.

The following result provides an equivalent characterization of relative Rota-Baxter operators.

\begin{prop}
    Let $(A, ~ \! \cdot ~ \! , \{ ~, ~, ~ \}, \{ \! \! \{ ~, ~, ~ \} \! \! \})$ be an associative-Yamaguti algebra and $(M, ~ \! \cdot ~ \! , \{ ~, ~, ~ \}, \{ \! \! \{ ~, ~, ~ \} \! \! \})$ be a representation of it. Then a linear map $R: M \rightarrow A$ is a relative Rota-Baxter operator if and only if the graph $\mathrm{Gr}(R) = \{ (R(u), u) ~ \! | ~ \! u \in M \}$ is a subalgebra of the semidirect product associative-Yamaguti algebra $(A \oplus M, ~ \! \cdot_\ltimes ~ \! , \{ ~, ~ , ~ \}_\ltimes , \{ \! \! \{ ~ , ~ , ~ \} \! \! \}_\ltimes)$.
   % \begin{itemize}
    %    \item[(i)] $R$ is a relative Rota-Baxter operator,
    %    \item[(ii)] ,
    %    \item[(iii)] the map $\widetilde{R} : A \oplus M \rightarrow A \oplus M$ defined by $\widetilde{R}(a, u) = (R (u), 0)$, is a Nijenhuis operator on the semidirect product $(A \oplus M, ~ \! \cdot_\ltimes ~ \! , \{ ~, ~ , ~ \}_\ltimes , \{ \! \! \{ ~ , ~ , ~ \} \! \! \}_\ltimes).$
    %\end{itemize}
\end{prop}

\begin{thm}\label{thm-relrba-dendy}
    Let $(A, ~ \! \cdot ~ \! , \{ ~, ~ , ~ \} , \{ \! \! \{ ~, ~, ~ \} \! \! \})$ be an associative-Yamaguti algebra and  $(M, ~ \! \cdot ~ \! , \{ ~, ~ , ~ \} , \{ \! \! \{ ~, ~, ~ \} \! \! \})$  be a representation of it. Suppose $R: M \rightarrow A$ is a relative Rota-Baxter operator. For any $u, v, w \in M$, we define
    \begin{align*}
        & \qquad \qquad  u \prec v := u \cdot R(v), \quad u \succ v := R(u) \cdot v, \\
         &\{ u, v, w \}_{[1]} := \{ u, R(u), R(v) \}, \quad \{ u, v, w \}_{[2]} := \{ R(u), v, R(w) \}, \quad \{ u, v, w \}_{[3]} := \{ R(u), R(v), w \},\\
          &\{ \! \! \{ u, v, w \} \! \! \}_{[1]} :=  \{ \! \! \{ u, R(u), R(v) \} \! \! \}, \quad \{ \! \! \{ u, v, w \} \! \! \}_{[2]} := \{ \! \! \{ R(u), v, R(w) \} \! \! \}, \quad \{ \! \! \{ u, v, w \} \! \! \}_{[3]} := \{ \! \! \{ R(u), R(v), w \} \! \! \}.
    \end{align*}
    Then $(M, \prec, \succ, \{ ~, ~ , ~ \}_{[1]}, \{ ~, ~, ~ \}_{[2]}, \{ ~, ~, ~ \}_{[3]} , \{ \! \! \{ ~, ~ , ~ \} \! \! \}_{[1]},  \{ \! \! \{ ~, ~ , ~ \} \! \! \}_{[2]},  \{ \! \! \{ ~, ~ , ~ \} \! \! \}_{[3]})$ is a dendriform-Yamaguti algebra.
\end{thm}

\begin{proof}
    For any $u, v, w \in M$, we have
    \begin{align*}
        &(u \prec v) \prec w - u \prec (v \prec w + v \succ w) + \{ u, v, w \}_{[1]} - \{ \! \! \{ u, v, w \} \! \! \}_{[1]} \\
        &= (u \cdot R(v)) \cdot R(w)- u \cdot R ( v \cdot R(w) + R(v) \cdot w ) + \{ u, R(v) , R(w) \} - \{ \! \! \{ u, R(v), R(w) \} \! \! \} \\
       & = (u \cdot R(v)) \cdot R(w)- u \cdot (R(v) \cdot R(w)) + \{ u, R(v) , R(w) \} - \{ \! \! \{ u, R(v), R(w) \} \! \! \} \stackrel{(\text{\ref{ay1}})}{=} 0,
    \end{align*}
    \begin{align*}
        &(u \succ v) \prec w - u \succ (v \prec w) + \{ u, v, w \}_{[2]} - \{ \! \! \{ u, v, w \} \! \! \}_{[2]} \\
        &= (R(u) \cdot v) \cdot R(w) - R(u) \cdot (v \cdot R(w)) + \{ R(u), v, R(w) \} - \{ \! \! \{ R(u), v, R(w) \} \! \! \}  \stackrel{(\text{\ref{ay1}})}{=} 0,
    \end{align*}
    and
    \begin{align*}
        &(u \prec v + u \succ v) \succ w - u \succ (v \succ w) + \{ u, v, w \}_{[3]} - \{ \! \! \{ u, v, w \} \! \! \}_{[3]} \\
        &= R (u \cdot R(v) + R(u) \cdot v) \cdot w - R(u) \cdot (R(v) \cdot w) + \{ R(u), R(v) , w \} - \{ \! \! \{ R(u), R(v), w \} \! \! \} \\
        &= (R(u) \cdot R(v)) \cdot w - R(u) \cdot (R(v) \cdot w) + \{ R(u), R(v), w \} - \{ \! \! \{ R(u), R(v), w \} \! \! \}  \stackrel{(\text{\ref{ay1}})}{=} 0.
    \end{align*}
    Hence, the identities (DY1A)-(DY1C) follow. Similarly, for any $u, v, w, s \in M$, we see that
    \begin{align*}
        \{ u \prec v, w, s \}_{[1]} =& \{ u \cdot R(v), R(w), R(s) \}  \stackrel{(\text{\ref{ay2}})}{=} \{ u, R(v) \cdot R(w) , R(s) \}  \\
        & \qquad =  \{ u, R(R(v) \cdot w + v \cdot R(w) ), R(s) \} = \{ u, v \prec w + v \succ w , s \}_{[1]}, 
    \end{align*}
    \begin{align*}
        \{ u \succ v, w, s \}_{[1]} = \{ R(u) \cdot v, R(w), R(s) \} \stackrel{(\text{\ref{ay2}})}{=} \{ R(u), v \cdot R(w), R(s) \}  = \{ u, v \prec w, s \}_{[2]},
    \end{align*}
    \begin{align*}
        \{ u \prec v + u \succ v, w, s \}_{[2]} =~& \{  R (R(u) \cdot v + u \cdot R(v)) , w, R(s)  \} = \{ R(u) \cdot R(v), w, R(s) \} \\ 
       & \qquad \stackrel{(\text{\ref{ay2}})}{=} \{ R(u) , R(v) \cdot w, R(s) \} 
        = \{ u, v \succ w, s \}_{[2]},
    \end{align*}
    \begin{align*}
        \{ u \prec v + u \succ v, w, s \}_{[3]} =& \{  R (R(u) \cdot v + u \cdot R(v)) , R(w), s \} = \{ R(u) \cdot R(v), R(w), s \} \\
        & \qquad \stackrel{(\text{\ref{ay2}})}{=} \{ R(u), R(v) \cdot R(w), s \}  = \{ u, v \prec w + v \succ w , s \}_{[3]}.
    \end{align*}
    Thus, by using (\ref{ay2}), we verified the identities (DY2A)-(DY2D) of a dendriform-Yamaguti algebra. In the same way, by using (AYi), one may show the identities (DYiA)-(DYiD), for $i=3, 4,5,6$. Also, by using (AYj), one may derive the identities (DYjA)-(DYjE), for $j = 7,8,9,10,11$. Hence, the result follows.
\end{proof}

In the previous theorem, we have shown that a relative Rota-Baxter operator induces a dendriform-Yamaguti algebra. The next result proves the converse. Namely, it shows that any dendriform-Yamaguti algebra can be obtained from a relative Rota-Baxter operator.

\begin{thm}\label{thm-last}
Let $(D, \prec, \succ, \{ ~, ~ , ~ \}_{[1]}, \{ ~, ~, ~ \}_{[2]}, \{ ~, ~, ~ \}_{[3]} , \{ \! \! \{ ~, ~ , ~ \} \! \! \}_{[1]},  \{ \! \! \{ ~, ~ , ~ \} \! \! \}_{[2]},  \{ \! \! \{ ~, ~ , ~ \} \! \! \}_{[3]})$ be a dendriform-Yamaguti algebra with the (total) associative-Yamaguti algebra $D_\mathrm{Tot} = (D, ~ \! \cdot_\mathrm{Tot} ~ \! , \{ ~, ~, ~ \}_\mathrm{Tot}, \{ \! \! \{ ~, ~, ~ \} \! \! \}_\mathrm{Tot})$. We define maps
\begin{align*}
    &\cdot : D_\mathrm{Tot} \otimes D \rightarrow D ~~ \text{ by } ~~ a \cdot b := a \succ b, \qquad \cdot : D \otimes D_\mathrm{Tot} \rightarrow D ~~ \text{ by } ~~ a \cdot b := a \prec b, \\
    & \qquad \qquad \{ ~, ~, ~ \} :  D_\mathrm{Tot} \otimes  D_\mathrm{Tot} \otimes D \rightarrow D ~~ \text{ by } ~~ \{ a, b, c \} := \{ a, b, c \}_{[3]},\\
    & \qquad \qquad \{ ~, ~, ~ \} :  D_\mathrm{Tot} \otimes D  \otimes  D_\mathrm{Tot} \rightarrow D ~~ \text{ by } ~~ \{ a, b, c \} := \{ a, b, c \}_{[2]}, \\
    & \qquad \qquad \{ ~, ~, ~ \} :   D  \otimes  D_\mathrm{Tot} \otimes D_\mathrm{Tot} \rightarrow D ~~ \text{ by } ~~ \{ a, b, c \} := \{ a, b, c \}_{[1]},\\
    & \qquad \qquad \{ \! \! \{ ~, ~, ~ \} \! \! \}:  D_\mathrm{Tot} \otimes  D_\mathrm{Tot} \otimes D \rightarrow D ~~ \text{ by } ~~ \{ \! \! \{ a, b, c \} \! \! \} := \{ \! \! \{  a, b, c \} \! \! \}_{[3]},\\
    & \qquad \qquad \{ \! \! \{ ~, ~, ~ \} \! \! \} :  D_\mathrm{Tot} \otimes D  \otimes  D_\mathrm{Tot} \rightarrow D ~~ \text{ by } ~~ \{ \! \! \{ a, b, c \} \! \! \} := \{ \! \! \{  a, b, c \} \! \! \}_{[2]}, \\
    & \qquad \qquad \{ \! \! \{ ~, ~, ~ \} \! \! \} :   D  \otimes  D_\mathrm{Tot} \otimes D_\mathrm{Tot} \rightarrow D ~~ \text{ by } ~~ \{ \! \! \{ a, b, c \} \! \! \} := \{ \! \! \{ a, b, c \} \! \! \}_{[1]}.
\end{align*}

\begin{itemize}
    \item[(i)] With the above action maps, $D$ is a representation of the associative-Yamaguti algebra $D_\mathrm{Tot}$.
    \item[(ii)] Then the identity map $\mathrm{Id} : D \rightarrow D_\mathrm{Tot}$ is a relative Rota-Baxter operator. Moreover, the induced dendriform-Yamaguti algebra structure on the space $D$ coincides with the prescribed one.
\end{itemize}
\end{thm}

\begin{proof}
    (i) This part follows from direct calculations. One can easily see that each identity  of the dendriform-Yamaguti algebra (there are a total $58$ identities) corresponds to exactly one condition of the representation (here also there are a total $58$ conditions).

    (ii) For any $a, b \in D$, we have
    \begin{align*}
        \mathrm{Id} (a) \cdot_\mathrm{Tot} \mathrm{Id} (b) = a \prec b + a \succ b = \mathrm{Id} \big(  {\mathrm{Id} (a)} \cdot b ~ \! +~ \! a \cdot \mathrm{Id}(b)  \big).
    \end{align*}
    Similarly, one can show that
    \begin{align*}
        \{ \mathrm{Id} (a), \mathrm{Id} (b), \mathrm{Id} (c) \}_\mathrm{Tot} =~& \mathrm{Id} \big(   \{ \mathrm{Id} (a), \mathrm{Id} (b), c \} + \{ \mathrm{Id} (a), b, \mathrm{Id} (c) \}  + \{ a, \mathrm{Id} (b), \mathrm{Id} (c) \}  \big),\\
        \{ \! \! \{ \mathrm{Id} (a), \mathrm{Id} (b), \mathrm{Id} (c) \} \! \! \}_\mathrm{Tot} =~& \mathrm{Id} \big(   \{ \! \! \{ \mathrm{Id} (a), \mathrm{Id} (b), c \} \! \! \} + \{ \! \! \{ \mathrm{Id} (a), b, \mathrm{Id} (c) \} \! \! \}  + \{ \! \! \{ a, \mathrm{Id} (b), \mathrm{Id} (c) \} \! \! \}  \big),
    \end{align*}
    for $a, b, c \in D$. Hence $\mathrm{Id} : D \rightarrow D_\mathrm{Tot}$ is a relative Rota-Baxter operator. Finally, if the induced dendriform-Yamaguti structure is $(D, \prec', \succ', \{ ~, ~ , ~ \}'_{[1]},  \{ ~, ~ , ~ \}'_{[2]},  \{ ~, ~ , ~ \}'_{[3]},  \{ \! \! \{ ~, ~ , ~ \} \! \! \}'_{[1]},  \{ \! \! \{ ~, ~ , ~ \} \! \! \}'_{[2]},  \{ \! \! \{ ~, ~ , ~ \} \! \! \}'_{[3]})$ then
    \begin{align*}
        a \prec' b = a \cdot \mathrm{Id} (b) = a \prec b ~~~~ \text{ and } ~~~~ a \succ' b = \mathrm{Id}(a) \cdot b = a \succ b, \text{ for } a, b \in D.
    \end{align*}
    Similarly, $\{ ~, ~, ~ \}'_{[i]} = \{ ~, ~, ~ \}_{[i]}$ and $\{ \! \! \{ ~, ~, ~ \} \! \! \}'_{[i]} = \{ \! \! \{ ~, ~, ~ \} \! \! \}_{[i]}$, for all $i = 1,2,3$. Hence, the result follows.
\end{proof}

\noindent {\bf Concluding remarks and further discussions.} This paper introduces associative-Yamaguti algebras as the associative analogue of Lie-Yamaguti algebras. The $(2,3)$-cohomology group of an associative-Yamaguti algebra is also introduced, and applications to formal one-parameter deformations and abelian extensions are obtained. Finally, Yamaguti multiplications and dendriform-Yamaguti algebras are considered, and their relations with relative Rota-Baxter operators are discussed. Since the notion of an associative-Yamaguti algebra is closely connected to associative algebras, associative triple systems, diassociative algebras and Lie-Yamaguti algebras, this new algebraic structure opens several interesting questions. First, one may ask about the full cochain complex of an associative-Yamaguti algebra generalizing the $(2,3)$-cohomology group considered in this paper and the well-known cochain complex of an associative triple system \cite{carlsson}. Further, in all the identities of an associative-Yamaguti algebra, since the variables stay in the same order, it is worth noting to find a cup product and a graded Lie bracket in the corresponding cochain complex (generalizing the well-known cup product and the graded Lie bracket on the Hochschild cochain complex of an associative algebra). To our knowledge, the existence of a cup product or a graded Lie bracket on the complex of an associative triple system is not even known.

\medskip

The notion of associative-Yamaguti algebra introduced here also motivates us to consider a generalization of an associative triple system. Explicitly, we say a tuple $(A, \{ ~, ~, ~ \}, \{ \! \! \{ ~, ~, ~ \} \! \! \})$ consisting of a vector space $A$ with linear operations $\{ ~, ~, ~ \}, \{ \! \! \{ ~, ~, ~ \} \! \! \} : A^{\otimes 3} \rightarrow A$ a {\em weak associative triple system} if only the identities (\text{\ref{ay7}}) and (\text{\ref{ay9}}) hold. A diassociative algebra $(D, \dashv, \vdash)$ yields a {weak associative triple system}, where the operations $\{ ~, ~, ~ \}$ and $ \{ \! \! \{ ~, ~, ~ \} \! \! \}$ are given by (\ref{diass-assy2}), (\ref{diass-assy3}). On the other hand, a weak associative triple system $(A, \{ ~, ~, ~ \}, \{ \! \! \{ ~, ~, ~ \} \! \! \})$ gives rise to a diassociative algebra structure on the tensor product $A \otimes A$ with the operations
\begin{align*}
    (a \otimes b) \dashv (c \otimes d) =  a \otimes \{ b, c, d \} \quad \text{ and } \quad 
    (a \otimes b) \vdash (c \otimes d) =  \{ \! \! \{ a, b, c \} \! \! \} \otimes d,
\end{align*}
for $a \otimes b,~ \! c \otimes d \in A \otimes A$. We believe that a weak associative triple system is a much better algebraic structure than an associative triple system associated with diassociative algebras. Thus, at this place, one may also look for the cohomology of a weak associative triple system (generalizing the cohomology of \cite{carlsson}) and find its connection to the cohomology of the corresponding diassociative algebra.

\medskip

A traditional question is to find out a suitable group-like object whose differentiation yields a Lie-Yamaguti algebra structure. It is well-known that a Lie digroup \cite{kinyon-self} serves as the group-like object for Leibniz algebras. Since Leibniz algebras form subclasses of Lie-Yamaguti algebras, it is highly expected that Lie digroups form subclasses of the group-like object corresponding to Lie-Yamaguti algebras (which one may call {\em Lie-Yamaguti groups}). Although the description of the Lie-Yamaguti group is not in hand, we expect that some of the identities (\text{\ref{ay1}})-(\text{\ref{ay11}}) of an associative-Yamaguti algebra will appear in the definition of a Lie-Yamaguti group. This can be justified by the following diagrammatic observation:
\begin{align*}
    \xymatrix{
    \substack{\text{diassociative} \\ \text{algebra} } \ar[d]_{\mathrm{Theorem ~} \ref{diass-assy}} \ar[rr]^{\text{group-like}}_{\text{object}}&  & \text{Lie digroup} \ar@{-->}[d] \ar[rr]^{\text{differentiation}} & & \text{Leibniz algebra} \ar[d]^{\mathrm{Example ~} \ref{exam-leib-liey}} \\
    \substack{\text{associative-Yamaguti} \\ \text{algebra} } \ar[rr]^{\text{group-like}}_{\text{object}} & & \substack{??? \\ \text{Lie-Yamaguti group}} \ar[rr]_{\text{differentiation}} & & \substack{\text{Lie-Yamaguti} \\ \text{algebra}}.
    }
\end{align*}

\medskip

An associative-Yamaguti algebra $(A, ~ \! \cdot ~ \! , \{ ~, ~, ~ \} , \{ \! \! \{ ~, ~ , ~ \} \! \! \})$ is said to be {\em commutative} if $a \cdot b = b \cdot a$, $\{ a, b, c \} = \{ b, a, c \}$ and $\{ \! \! \{ a, b, c \} \! \! \} = \{ \! \! \{ a, c, b \} \! \! \}$, for all $a, b, c \in A$. The last two conditions can be equivalently understood by  $\sigma_{a, b} = \sigma_{b, a}$ and $\tau_{a, b} = \tau_{b, a}$, for all $a, b \in A$. In this regard, one may look for the appropriate notion of a Poisson-Yamaguti algebra. Precisely, it would be given by a commutative associative-Yamaguti algebra endowed with a Lie-Yamaguti algebra structure that satisfies certain compatibilities that need to be figured out. Further, one may find the possible connections between Poisson-Yamaguti algebras and Poisson-dialgebras as introduced by Loday \cite{loday} (in the same way Lie-Yamaguti algebras are related to Leibniz algebras, and associative-Yamaguti algebras are related to diassociative algebras).

\medskip

\noindent  {\bf Acknowledgements.} The author would like to thank the Department of Mathematics, IIT Kharagpur, for providing the beautiful academic atmosphere where the research has been conducted.

\medskip

\noindent {\bf Data Availability Statement.} Data sharing does not apply to this article as no new data were created or analyzed in this study.

\end{document}